\documentclass[12pt,tbtags,leqno]{amsart}

\usepackage[T1]{fontenc}
\usepackage[latin1]{inputenc}
\usepackage[english]{babel}
\usepackage{amssymb}
\usepackage{amsthm}
\usepackage{amsmath}
\usepackage{verbatim}
\usepackage{color}
\usepackage{hyperref}

\numberwithin{equation}{section}

\newcommand{\HH}{\mathcal{H}}
\newcommand{\HK}{\mathcal{K}}

\newcommand{\HE}{\mathcal{E}}
\newcommand{\HF}{\mathcal{F}}

\newcommand{\HD}{\mathcal{D}}
\newcommand{\HB}{\mathcal{B}}

\newcommand{\dB}{\partial \mathbb B_d}

\newcommand{\D}{\mathbb{D}}

\newcommand{\C}{\mathbb{C}}
\newcommand{\N}{\mathbb{N}}
\newcommand{\R}{\mathbb{R}}

\newcommand{\la}{\langle}
\newcommand{\ra}{\rangle}
\newcommand{\Bd}{\mathbb B_d}

\newcommand{\Hol}{\operatorname{Hol}}

\def\HE{{\mathcal E}}
\def\HF{{\mathcal F}}

\def\om{\omega}
\def\vhat{\hat{v}}

\newcommand{\Ker}[1]{\mathsf{Ker}~}

\theoremstyle{plain}
\newtheorem{theorem}{Theorem}[section]
\newtheorem{lemma}[theorem]{Lemma}

\newtheorem{prop}[theorem]{Proposition}
\newtheorem{corollary}[theorem]{Corollary}

\newtheorem{definition}[theorem]{Definition}

\theoremstyle{definition}
\newtheorem{example}[theorem]{Example}

\newcommand{\Mult}{\operatorname{Mult}}
\newcommand\cH{{\mathcal H}}
\newcommand\cK{{\mathcal K}}
\newcommand\cT{{\mathcal T}}
\newcommand\cM{{\mathcal M}}
\newcommand\cB{{\mathcal B}}
\newcommand{\bB}{\mathbb B}
\newcommand{\bC}{\mathbb C}
\newcommand{\bD}{\mathbb D}
\newcommand{\bN}{\mathbb N}
\newcommand{\bT}{\mathbb T}
\newcommand{\ol}[1]{\overline{#1}}
\newcommand{\diag}{\operatorname{diag}}
\newcommand{\cE}{{\mathcal E}}
\newcommand{\cF}{{\mathcal F}}

\begin{document}

\bibliographystyle{plain}

\thanks{}

\date{\today}

\title[Radially weighted Besov spaces and the Pick property]{Radially weighted Besov spaces and the Pick property}
\author[A. Aleman]{Alexandru Aleman}
\address{Lund University, Mathematics, Faculty of Science, P.O. Box 118, S-221 00 Lund, Sweden}
\email{alexandru.aleman@math.lu.se}

\author[M. Hartz]{Michael Hartz}
\address{Department of Mathematics, Washington University in St. Louis, One Brookings Drive,
St. Louis, MO 63130, USA}
\email{mphartz@wustl.edu}
\thanks{M.H. was partially supported by a Feodor Lynen Fellowship}

\author[J. M\raise.5ex\hbox{c}Carthy]{John E. M\raise.5ex\hbox{c}Carthy}
\address{Department of Mathematics, Washington University in St. Louis, One Brookings Drive,
St. Louis, MO 63130, USA}
\email{mccarthy@wustl.edu}
\thanks{J.M. was partially supported by National Science Foundation Grant DMS 1565243}

\author[S. Richter]{Stefan Richter}
\address{Department of Mathematics, University of Tennessee, 1403 Circle Drive, Knoxville, TN 37996-1320, USA}
\email{srichter@utk.edu}


\begin{abstract} For $s\in \R$ the weighted Besov space on the unit ball $\Bd$ of $\C^d$ is defined by
$$B^s_\om=\{f\in \Hol(\Bd): \int_{\Bd}|R^sf|^2 \om dV<\infty\}.$$
Here $R^s$ is a power of the radial derivative operator $R= \sum_{i=1}^d z_i\frac{\partial}{\partial z_i}$,
 $V$ denotes Lebesgue measure, and $\om$ is a radial
 weight function not supported on any ball of radius $< 1$.

 Our results imply that for all  such weights $\om$ and $\nu$,
 every bounded column multiplication operator $B^s_\om \to B^t_\nu \otimes \ell^2$ induces a bounded row multiplier $B^s_\om\otimes \ell^2 \to B^t_\nu$.
Furthermore we show that  if a weight $\om$ satisfies that for some $\alpha >-1$ the ratio $\om(z)/(1-|z|^2)^\alpha$ is nondecreasing for $t_0<|z|<1$, then
 $B^s_\om$ is a complete Pick space, whenever $s\ge (\alpha+d)/2$.
\end{abstract}

\maketitle


\section{Introduction}
Let $d\in \mathbb N$. In this paper we will address certain questions about functions and multipliers in  weighted Besov Hilbert spaces of analytic functions in the unit ball $\Bd=\{z\in \C^d: |z|<1\}$. In particular, we will show that results about multipliers in standard and  Bekoll\'{e} weighted Besov spaces of \cite{OrtFab} and \cite{CasFab_Bekolleweights} extend to hold for all  radial weights, and we will provide simple, but general conditions on radial weight functions $\om$ that imply that all  results  of \cite{AlHaMcCRiWeakProd} can be applied to such a weighted Besov space.

We will use $V$ to denote Lebesgue measure on $\C^d$ restricted to $\Bd$, normalized so that $V(\Bd)=1$.
 A non-negative integrable function $\om$ on $\Bd$ is called a radial weight, if for each $0<r<1$ the value $\om(rz)$ is independent of $z \in \dB$
 and the non-degeneracy condition
 \begin{equation}
 \label{eqjm1}
 \int_{|z|>r} \om dV>0 \quad {\rm for\ each\ } 0<r<1
 \end{equation}
 holds.
 It  is easily checked that for radial weights the weighted Bergman space $L^2_a(\om)=L^2(\om dV)\cap \Hol(\Bd)$ is closed in $L^2(\om dV)$, and that point evaluations $f\to f(z)$ are bounded on $L^2_a(\om)$ for each $z\in \Bd$.

We now fix a radial weight $\om$. Then we have $$\|f\|^2_{L^2_a(\om)}=\int_{\Bd}|f|^2 \om dV= \sum_{n\ge 0} \|f_n\|^2_{L^2_a(\om)},$$ where $f=\sum_{n\ge 0} f_n$ is the decomposition of the analytic function $f$ into a sum of homogeneous polynomials $f_n$ of degree $n$. We associate a one-parameter family of weighted Besov spaces $\{B^s_\om\}_{s\in \R}$ with $\om$ as follows:
 \begin{equation}\label{BesovDefinition} \|f\|^2_{B^s_\om}=\|\om\|_{L^1(V)}|f(0)|^2 + \sum_{n=1}^\infty n^{2s}\|f_n\|^2_{L^2_a(\om)}\end{equation}
 $$ B^s_\om=\{f\in \Hol(\Bd): \|f\|^2_{B^s_\om}<\infty\}.$$

Let $R=\sum_{i=1}^d z_i\frac{\partial}{\partial z_i}$ denote the radial derivative operator, then $Rf= \sum_{n\ge 1} nf_n$. More generally, for each nonzero $s\in \R$ we may consider the "fractional" transformation $R^s:\sum_{n\ge 0}  f_n \to \sum_{n\ge 1} n^s f_n$. It is thus clear that
\begin{align*}B^s_\om&=\{f\in \Hol(\Bd): R^s f \in L^2_a(\om)\},\\ \|f\|^2_{B^s_\om}&=\|\om\|_{L^1(V)}|f(0)|^2 + \int_{\Bd}|R^s f|^2\om dV.\end{align*}
 One checks that \eqref{eqjm1} implies that each $B^s_\om$ is a Hilbert space, and point evaluations for all points in $\Bd$ are bounded. A space $\HH$ of analytic functions that occurs as one of the spaces $B^s_\om$ for a radial weight $\om$ and some $s\in \R$ will be called a weighted Besov space.

If $\om(z)=1$, $s\in \R$, and $f\in \Hol(\Bd)$,  then $f\in B^s_\om$ if and only if $R^sf\in L^2_a$, the unweighted Bergman space. Thus, in this case the collection $B^s_\om$ consists of standard weighted Bergman or Besov spaces. We have $B^{d/2}_{{\mathbf 1}}=H^2_d$, the Drury-Arveson space, $B^{1/2}_{{\mathbf 1}}=H^2(\dB)$, the Hardy space of the Ball, and for $s<1/2$ we obtain the weighted Bergman spaces $B^s_{\mathbf 1}= L^2_a((1-|z|^2)^{-2s}dV)$, where all equalities are understood to mean equality of spaces with equivalence of norms. These spaces have been extensively studied in the literature. We refer the reader to \cite{ZhaoZhu}, where the $L^p$-analogues of these spaces were considered as well. If $d=1$ and $s=1$, then $B^1_{\mathbf 1}=D$, the classical Dirichlet space of the unit disc. More generally, if $d=1$ and $s>1/2$, then these spaces are have been referred to as Dirichlet-type spaces, see \cite{BrownShields}.

If $\om(z)=(1-|z|^2)^{\alpha}$ for some $\alpha >-1$, then $\om$ is called a standard weight, and we obtain the same spaces as for $\om_0=1$, but with a shift in indices: $B^s_{\om}= B^{s-\frac{\alpha}{2}}_{\mathbf 1}$. This can be verified by using polar coordinates  and the asymptotics $\int_0^1 t^n(1-t)^\alpha dt = \frac{\Gamma(n+1)\Gamma(\alpha+1)}{\Gamma(n+\alpha+2)}\approx n^{-\alpha-1}$, which follows e.g. from Stirling's formula. We refer the reader to Section \ref{section_weights_basics}  of the current paper for more detail on further calculations of this type.

Observe that for standard weights $\om$ the spaces $B^s_\om$ are weighted Bergman spaces for all $s\le 0$. More generally, the following will be Theorem \ref{omxTheorem}.

\begin{theorem}\label{IntroomxTheorem} Let $\om$ be a radial weight, let  $s> 0$, and  for $ z\in \Bd$ define $$\om_s(z)=
  \frac{1}{d} |z|^{2 - 2 d} \int_{ |w|\ge |z|} \frac{(|w|^2-|z|^2)^{2s-1}}{\Gamma(2s)} \ \om(w) dV(w).$$
  Then  $\om_s$ is a weight and $B^t_{\omega} = B^{t+s}_{\omega_s}$ with equivalence of norms for all $t \in \mathbb R$. In particular, $L^2_a(\om_s)=B^{-s}_\om$ with equivalence of norms.
\end{theorem}

One checks that for all $s\le 0$ and all  radial weights $\om$, we have  $\Mult(B^s_\om)=H^\infty.$  Here $H^\infty$ denotes the bounded analytic functions on $\Bd$, and  $$\Mult(\HB)=\{\varphi\in \Hol(\Bd): \varphi f \in \HB \text{ for all }f\in \HB\}$$ denotes the multiplier algebra of $\HB$.

In this paper we are interested in $\Mult(B^s_\om)$ for $s> 0$. In general in those cases it turns out that $\Mult(B^s_\om)$ is a proper subset of $H^\infty$, but it is worthwhile to note that there are  radial weights $\om$ such that $B^s_\om=L^2_a(\mu_s)$ for each $s\in \R$ for some weight $\mu_s$ and hence $\Mult(B^s_\om)=H^\infty$ holds for all $s\in \R$. The weight $\om(z)=e^{\frac{-1}{1-|z|^2}}$ is an example of a weight where this happens. Indeed, in this case for each positive integer $N$ the function $(1-|z|^2)^{-4N}\om(z)$ is also integrable, and in Example \ref{ExpExample} we will show that $R^N f \in L^2_a(\om)$ if and only if $$\int_{\Bd} |f|^2 (1-|z|^2)^{-4N}\om dV<\infty.$$ By Theorem \ref{IntroomxTheorem} this implies that $B^s_\om$ is a weighted Bergman space for each $s\le N$, and since $N$ was arbitrary it follows that the same is true for all $s\in \R$.

%
If $\HH\subseteq \Hol(\Bd)$ is a Hilbert function space and if $\HE$ is an auxiliary Hilbert space, then the identification of  elementary tensors of the type $f \otimes x$, $f\in \HH, x\in \HE$ with  $\HE$-valued functions $f(\cdot)x$ extends to define a Hilbert space $\HH(\HE)$ of $\HE$-valued analytic functions on $\Bd$ that is isomorphic to $\HH\otimes \HE$. If $\HH$ and $\HK$ are two Hilbert spaces of analytic functions on $\Bd$ and if $\HE$ and $\HF$ are auxiliary Hilbert spaces, then $\Mult(\HH(\HE),\HK(\HF))$ will denote the multipliers from $\HH(\HE)$ to $\HK(\HF)$, i.e. those functions $\Phi:\Bd \to \HB(\HE,\HF)$ such that $F\to M_\Phi F, (M_\Phi F)(z)=\Phi(z)F(z)$ defines a bounded linear transformation from $\HH(\HE)$ to $\HK(\HF)$. We will write $\Mult(\HH,\HK)= \Mult(\HH(\C),\HK(\C))$ for the scalar-valued multipliers.

In the paper \cite{AlHaMcCRiWeakProd}, an important role was played by the multiplier inclusion condition. For a weighted Besov space $B^N_\om$, where $N \in \mathbb N$, this condition means that
\begin{equation*}
  \Mult(B_\omega^N,B_\omega^N(\ell^2)) \subseteq
  \Mult(B_\omega^{N-1},B_\omega^{N-1}(\ell^2)) \subseteq \cdots
  \subseteq \Mult(B^0_\omega, B^0_\omega(\ell^2))
\end{equation*}
with continuous inclusions.
We established this condition for the Drury--Arveson space and a few other standard weighted Besov spaces using an
elementary method. It is also possible to use the complex method of interpolation to establish inclusions of multiplier spaces.
 Indeed, if $s,t, \alpha \in \R$ with $s\le t$ and $\alpha\ge 0$, then it is shown in \cite{CasFab_Bekolleweights} that for Bekoll\'{e}-Bonami  weights $\om$ one has $$\Mult(B^{t+\alpha}_\om,B^{s+\alpha}_\om)\subseteq  \Mult(B^t_\om,B^s_\om).$$ Note that for Bekoll\'{e}-Bonami weights that are not necessarily radial the following definition is used for the weighted Besov space
 $$B^s_\om=\{f\in \Hol(\Bd): R^Nf\in L^2_a((1-|z|^2)^{2(N-s)}\om(z))\},$$ where $N$ is any non-negative integer $\ge s$. For radial weights satisfying a Bekoll\'{e}-Bonami condition this coincides with the definition used here since in that case $(1-|z|^2)^{2(N-s)}\om \approx \om_{N-s}$, see e.g. Lemmas \ref{vxdoublingweights} and \ref{BekolleAndDoubling}.

In this paper, we use a third method to establish a general result about inclusions of multiplier spaces of unitarily invariant Hilbert function spaces
on $\bB_d$,
using the fact that multiplication operators are triangular with respect to the common orthogonal basis of monomials.
In particular, we obtain the following theorem, which shows that the multiplier inclusion condition holds whenever $\om$ is a radial weight. It is proved in Corollary~\ref{cor:Besov_rectangular_multiplier_inclusion} (also see Corollary~\ref{unrestricted}).

\begin{theorem}
  \label{thm:mult_inclusion_intro}
  Let $\omega$ and $\nu$ be  radial weights in $\bB_d$ and let $s,t,s',t' \in \mathbb R$ with $t \le s$ and $t'-s' \le t-s$.
  Then for any pair $\cE,\cF$ of separable Hilbert spaces,
\begin{equation*}
  \Mult(B^s_\om(\cE), B^{s'}_{\nu}(\cF)) \subseteq \Mult(B^{t}_\om(\cE), B^{t'}_{\nu}(\cF))
\end{equation*}
and the inclusion is contractive.
\end{theorem}

Given a sequence $\Phi = \{\varphi_1,\varphi_2,\ldots \} \subseteq\Mult(\HH,\HK)$ of
multipliers, we can consider the
column operator  $\Phi^C: h\to (\varphi_1h, \varphi_2 h,...)^T$ and the row operator $\Phi^R:(h_1,h_2,...)^T \to \sum_{i\ge 1}\varphi_ih_i$. Here for ease of writing  we have used $(h_1,...)^T$ to denote the transpose of a row vector. We write $M^C(\HH,\HK)$ for the set of those sequences $\Phi$ whose
column operator $\Phi^C$ is bounded, that is, $\Phi^C \in \Mult(\HH,\HK(\ell_2))$. Similarly,
let $M^R(\HH,\HK)$ denote all sequences $\Phi$ for which the row operator is bounded, i.e. $\Phi^R \in \Mult(\HH(\ell_2),\HK)$.
We will abbreviate the notations to $M^R(\HH)$ and $M^C(\HH)$, if $\HH=\HK$.

 Trent showed that for the Dirichlet space $D$ of the unit disc $\D\subseteq \C$ one has the continuous inclusion $M^C(D)\subseteq M^R(D)$ and the norm of the inclusion is at most $\sqrt{18}$, see Lemma 1 of \cite{TrentCorona}.
 The results in \cite{AlHaMcCRiWeakProd} establish that $M^C(\HH) \subseteq M^R(\HH)$ for certain standard weighted Besov
 spaces $\HH$ including the Drury--Arveson space.
 Using Theorem \ref{thm:mult_inclusion_intro}, we now obtain a more general result, which is Theorem \ref{thm:BCBR_general}.

\begin{theorem}
  \label{thm:BCBR_intro}
  Let $\om$ and $\nu$ be  radial weights in $\bB_d$, and let $s, t \in \R$. Then
  \begin{equation*}
    M^C(B^s_\om, B^t_\nu) \subseteq M^R(B^s_\om,B^t_\nu)
  \end{equation*}
  and the inclusion is continuous.
\end{theorem}

 It is known and easy to verify that $M^C(L^2_a(\om))=M^R(L^2_a(\om))=H^\infty(\ell_2)$, where
$$H^\infty(\ell_2)=\{(\varphi_1,\varphi_2,...): \varphi_j\in H^\infty \text{ and } \sup_{z\in \Bd} \sum_j |\varphi_j(z)|^2 < \infty\}.$$

One application of Theorem \ref{thm:BCBR_intro} is to provide
another proof of the characterization of interpolating sequences established in \cite{AlHaMcCRiInter} in the case of radially weighted
Besov spaces with the complete Pick property. The proof in \cite{AlHaMcCRiInter} uses the Marcus--Spielman--Srivastava theorem \cite{MarcusSpielmanSrivastava}, but as explained in Remark 3.7 in \cite{AlHaMcCRiInter}, this theorem can be avoided for spaces $\HH$ with the property that $M^C(\HH) \subseteq M^R(\HH)$.

A Hilbert function space $\HH$ is a Hilbert space of complex-valued functions on a set $X$ such that point evaluations for points in $X$ define continuous linear functionals on $\HH$.
 Every Hilbert function space $\HH$ has a reproducing kernel, i.e. a function $k:X \times X \to \C$ such that $f(w)=\la f, k_w\ra $ for all $w\in X$, where $k_w(z)=k(z,w)$.  We say that $k$ is normalized, if there is a $z_0\in X$ such that $k_{z_0}=1$.

 By a normalized complete Pick kernel we mean a normalized reproducing kernel of the form $k_w(z)=\frac{1}{1-u_w(z)}$, where $u_w(z)$ is positive definite, i.e. whenever $n\in \N, z_1,...,z_n \in X, $ and $a_1,..., a_n \in \C$ we have $\sum_{i,j}a_i \overline{a}_j u_{z_j}(z_i) \ge 0$.
 (Normally complete Pick kernels are defined intrinsically, but by the McCullough-Quiggin theorem they are precisely of this form. See \cite{mccul92, qui93, agmc_cnp}).

 An important example of such a complete Pick kernel is the Szeg{\H{o}} kernel $k_w(z)=(1-\overline{w}z)^{-1}$. It is the reproducing kernel for the Hardy space $H^2$ of the unit disc $\D$.
 Many properties of the Hardy space carry over to other spaces with complete Pick kernels---see \cite{AlHaMcCRiM}
 for some examples.
 We will say that a Hilbert function space $\HH$ is a complete Pick space, if there is an equivalent norm on the space such that the reproducing kernel for that norm is a normalized complete Pick kernel.
In  \cite{AlHaMcCRiM} it is proven that  complete Pick spaces  $\HH$
   are contained in the Smirnov class $N^+(\HH)$ associated with $\HH$, where
   $$N^+(\HH)=\{f=\frac{\varphi}{\psi}: \varphi, \psi \in \Mult(\HH), \psi \text{ cyclic in }\HH\}$$ and a multiplier $\psi$ is called cyclic if $\psi \HH$ is dense in $\HH$.

  It is known that for all $s\ge d/2$ the spaces $B^{s}_{{\mathbf 1}}$ are complete Pick spaces
  (this can be seen as in Corollary 7.41 of \cite{AgMcC}).
 In particular, for $d/2\le s < (d+1)/2$ the space  $B^{s}_{{\mathbf 1}}$ has reproducing kernel $ \frac{1}{(1-\la z, w\ra )^{d+1-2s}}$ (up to equivalence of norms), which can be seen to be a complete Pick kernel by consideration of the binomial series coefficients.
   On the other hand, if $s<d/2$, then $B^s_{\mathbf 1}$ is not a complete Pick space, because $B^s_{\mathbf 1}\nsubseteq N^+(B^s_{\mathbf 1})$. Indeed, in this case $B^s_{\mathbf 1}$ has a reproducing kernel of the type $\frac{1}{(1-\la z,w\ra)^\gamma}$ for some $\gamma >1$.  If $d=1$, then  $B^s_{\mathbf 1}$ is a weighted Bergman space, which will contain functions that are not in the Nevanlinna class, and hence cannot be ratios of multipliers. The same is true if $d>1$. In that case  the $d=1$ result implies that there are functions of the form $f(z_1,0,...,0)$ in $B^s_{\mathbf 1}$ that are not the ratio of two bounded functions.

   An observation that was shared years ago with us by Serguey Shimorin is that if the Cauchy dual of a space of functions in the unit disc is a weighted Bergman space, then the original space is a complete Pick space. An analogue of this holds for functions in $\Bd$ and for radially symmetric spaces we have worked that out in Lemma \ref{Pick kernel}. For many radially symmetric weighted Besov spaces that leads to a condition which is easy to check:

  \begin{theorem} \label{IntroPickBesov} Let $\alpha>-1$, $0\le r_0 <1$, and let $\om$ be a radial weight such that $\frac{\om(z)}{(1-|z|^2)^\alpha}$ is nondecreasing in $|z|$ for $r_0<|z|<1$. Then $B^s_\om$ is a complete Pick space for all $s\ge \frac{\alpha+d}{2}$.
  \end{theorem}

This will follow from Theorem \ref{PickandBekolle}, which holds for weights that satisfy a related, but weaker condition.

If $\HH$ and $\HK$ are two Hilbert function spaces on the same set, then we will write $\HH = \HK$ to mean that
$\HH$ and $\HK$ agree as vector spaces and their norms are equivalent, but not necessarily equal.
If $\|\cdot\|_1$ and $\|\cdot\|_2$ are two norms, then we will write $\|f\|_1 \approx \|f\|_2$ to denote that the norms are equivalent. Similarly, if $a_n, b_n \ge 0$, then $a_n \approx b_n$ will mean that there are constants $c,C>0$ such that $ca_n \le b_n \le Ca_n$ holds for all $n \in \N$.

The remainder of this paper is organized as follows. In Section \ref{section_weights_basics}, we collect basic facts
about radially weighted Besov spaces and then prove Theorem \ref{IntroomxTheorem}. In Section \ref{SectionMICradial}, we prove
several results about inclusions of multiplier algebras and of multiplier spaces. In particular, we show Theorems \ref{thm:mult_inclusion_intro} and
\ref{thm:BCBR_intro}. Section \ref{sec:weights_finer} is devoted to the study of several finer properties of weights.
In particular, we introduce weakly normal weights, which will be important in the proof of Theorem \ref{IntroPickBesov}.
Section \ref{SectionRadialPick} then contains the proof of Theorem \ref{IntroPickBesov}. In the final Section \ref{SectionFurther},
we use the methods developed in this paper to establish some additional properties of multipliers of weighted Besov spaces.

\section{Radially weighted Besov spaces and index shifts}
\label{section_weights_basics}

\subsection{Basics about radially weighted Besov spaces}
\label{subsec:basics}

Let $\omega$ be a radial weight on $\Bd$.  We will temporarily write $u_\om(r)=\om(r,0,...,0)$ if $r\in (0,1)$.
Let $\sigma$ be Lebesgue measure on $\dB$, normalized so that $\sigma(\dB)=1$. Then for any non-negative measurable function $h$ on $\Bd$ we have the change of variables
\begin{align*}\int_{\Bd}h\om dV &= \int_0^1 \left(\int_{\dB} h(rw)  d\sigma(w)\right) u_\om(r) 2d r^{2d-1}dr.\end{align*}
In particular, if $f\in \Hol(\Bd)$ with homogeneous expansion $f=\sum_{n=0}^\infty f_n$, then
\begin{align}\label{radialNorm} \int_{\Bd} |f|^2 \om dV &= \sum_{n=0}^\infty a_n(\om) \|f_n\|^2_{H^2(\dB)},\end{align}
where $$a_n(\om)= 2d\int_0^1 r^{2n+2d-1} u_\om(r) dr=  \int_0^1 t^{n} v(t) dt.$$
Here we used
 $v(t)$ is the product $ d \cdot  t^{d-1} \cdot u_\om(\sqrt{t})$ and note that $v \in L^1[0,1]$.
It is clear that this process can be reversed and any positive $L^1[0,1]$-function $v$ can be used as above to associate a function  $\om$ on $\Bd$.
The non-degeneracy condition \eqref{eqjm1} is equivalent to
\begin{equation}
\label{eqjm2}
\int_t^1 v(x)dx >0 \ \quad \text{ for all } t<1.
\end{equation}


We will say that a non-negative function $v \in L^1[0,1]$ is a weight if \eqref{eqjm2} holds.

The following elementary lemma about moments of  weights will be useful in several places.

\begin{lemma}
  \label{lem:weight_asymptotics}
  Let $v,w \in L^1[0,1]$ be two non-negative  weights such that $\lim_{t \nearrow 1}  \frac{v(t)}{w(t)} = 1$ (with the convention $0/0 = 1$). Then
  \begin{equation*}
    \lim_{n \to \infty} \frac{ \int_0^1 t^n v(t) \, dt }{\int_0^1 t^n w(t) \, dt} = 1.
  \end{equation*}
\end{lemma}

\begin{proof}
  By symmetry, it suffices to show that
  \begin{equation*}
    \limsup_{n \to \infty} \frac{ \int_0^1 t^n v(t) \, dt }{\int_0^1 t^n w(t) \, dt} \le 1.
  \end{equation*}
  To this end, let $r \in (0,1)$ be such that $\frac{v(t)}{w(t)}$ is finite for $t \in [r,1]$. Then
  \begin{equation*}
    \int_0^1 t^n w(t) \, d t \ge \int_r^1 t^n w(t) \, dt \ge r^{n / 2} \int_{\sqrt{r}}^1  w(t) \, dt,
  \end{equation*}
  where the last quantity is strictly positive by \eqref{eqjm2}.
  Moreover,
  \begin{align*}
    \int_0^1 t^n v(t) \, dt &= \int_0^r t^n v(t) \, dt + \int_r^1 t^n v(t) \, dt \\
    &\le r^n \int_0^1 v(t) \, dt + \sup_{x \in [r,1]} \frac{v(x)}{w(x)} \int_r^1 t^n w(t) \, dt.
  \end{align*}
  Therefore,
  \begin{equation*}
    \frac{ \int_0^1 t^n v(t) \, dt }{\int_0^1 t^n w(t) \, dt} \le \sup_{x \in [r,1]} \frac{v(x)}{w(x)} + r^{n/2} \frac{\int_0^1 v(t) \, dt}{\int_{\sqrt{r}}^1 w(t) \, dt },
  \end{equation*}
  so that
  \begin{equation*}
    \limsup_{n \to \infty} \frac{ \int_0^1 t^n v(t) \, dt }{\int_0^1 t^n w(t) \, dt} \le \sup_{x \in [r,1]} \frac{v(x)}{w(x)}.
  \end{equation*}
  This is true for all $r$ sufficiently close to $1$. The result now follows by taking the limit $r \nearrow 1$.
\end{proof}

Let now $\omega$ be a radial weight in $\bB_d$.
We will use the moments $a_n(\om)=\int_0^1t^n v(t) dt$ to express the norm of $B^s_\om$. For $f\in \Hol(\Bd)$ we will continue to write $f=\sum f_n$ for its expansion into a sum of homogeneous polynomials.

Let $s\in \R$. Comparison of (\ref{BesovDefinition}) and (\ref{radialNorm}) shows that  \begin{align}\label{Bsom-norm}  \|f\|^2_{B^s_\om}= a_0(\om)|f(0)|^2 + \sum_{n=1}^\infty n^{2s} a_n(\om) \|f_n\|^2_{H^2(\dB)}.\end{align}
Since $R^Nf=\sum_{n\ge 0} n^N f_n$ it is clear that for each $s\in \R$ we have $f\in B^s_\om$, if and only if $R^N f \in B^{s-N}_\om$.

We also remark that the reproducing kernel of $B_{\omega}^s$ is of the form
\begin{equation*}
  k_w(z) = \sum_{n=0}^\infty b_n \langle z,w \rangle^n,
\end{equation*}
where for $n\ge 1$
\begin{equation*}
  b_n = ||z_1^n||^{-2}_{B^s_{\omega}} = n^{-2 s} a_n(\omega)^{-1}\|z_1^n\|^{-2}_{H^2(\dB)}\approx  n^{-2 s+d-1} a_n(\omega)^{-1}.
\end{equation*}
It follows from Lemma \ref{lem:weight_asymptotics} that $\frac{\int_0^1 t^{n+1} v(t) dt}{\int_0^1 t^{n} v(t) dt} \to 1$ as $n\to \infty$ for any weight $v \in L^1[0,1]$. Hence, $\lim_{n \to 1} b_n / b_{n+1} = 1$. This condition is frequently useful in operator theoretic contexts. For instance, it implies that the tuple $(M_{z_1},\ldots,M_{z_d})$
of multiplication operators by the coordinate functions is essentially normal and has
essential Taylor spectrum $\partial \mathbb B_d$, see Theorem 4.5 of \cite{GHX}.

\subsection{Index shift}
\label{sec_v_x}

Recall from the Introduction that $B_{\mathbf{1}}^s = B_{\omega_{\alpha}}^{s + \frac{\alpha}{2}}$
for all $s \in \mathbb R$ and $\alpha > -1$, where
$\omega_\alpha(z) = (1 - |z|^2)^\alpha$ is a standard weight. We now introduce a generalization of
this procedure which will allow us to shift the index $s$ of the space $B_{\omega}^s$
for more general radial weights $\omega$.

We saw in Section \ref{subsec:basics} that by a change to polar coordinates any radial weight $\om$ on $\Bd$ is associated with a non-negative function $v\in L^1[0,1]$. More generally, let $\mu$ be a finite Borel measure on $[0,1]$. For $x>0$ consider
\begin{align*}\int_0^1 \int_{[t,1]}(s-t)^{x-1} d\mu(s) dt&= \int_{[0,1]} \int_0^s (s-t)^{x-1} dt d\mu(s)\\
&= \int_{[0,1]}\frac{s^{x}}{x}d\mu(s)<\infty.\end{align*}
Thus, for all $x>0$ we can define a non-negative $L^1[0,1]$-function $v_x$ by $$v_x(t)= \int_{[t,1]} \frac{(s-t)^{x-1}}{\Gamma(x)} d\mu(s), t \in [0,1).$$ Here $\Gamma(x)$ denotes the Gamma function.
  It is easy to check that the functions $v_x$ obey the semigroup law $(v_x)_y = v_{x + y}$
  for all $x,y > 0$. We also remark that if $v_1(t) > 0$ for all $r \in (0,1)$, then $v_x$ satisfies \eqref{eqjm2}
   for all $x > 0$.

  The following lemma will be used repeatedly. It will allow us to perform the desired
  index shift for $B^s_{\omega}$ (see Theorem \ref{omxTheorem} below).

\begin{lemma} \label{momentsByParts} Let $\mu$ be a finite positive Borel measure on $[0,1]$, for $x>0$ let $v_x$ be the function associated with $\mu$ as above, and assume that $v_1(t)>0$ for all $0\le t<1$.

 Then for each $x>0$ we have
$$ \lim_{n\to \infty} \frac{ n^{x} \int_0^1 t^nv_x(t) dt}{\int_{[0,1]} t^nd\mu} = 1.$$
\end{lemma}
\begin{proof} We start with the observation that for any integer $n>0$ we have $\int_0^1 t^{n-1} \left(\log(1/t)\right)^{x-1} dt = n^{-x} \Gamma(x)$. This can easily be verified with the substitution $t= e^{-\frac{u}{n}}$  (see \cite{XLi}, p.56). Next we define the auxiliary function
$$v^*_x(t)= \int_{[t,1]}\frac{\left(\log \frac{s}{t} \right)^{x-1}}{\Gamma(x)} d\mu(s).$$
An application of Fubini's theorem and the earlier observation shows that $$n^x \int_0^1 t^{n-1} v^*_x(t) dt=\int_{[0,1]} t^n d\mu(t), \ \ n=1,2,...$$
So in order to prove the Lemma, it suffices to show that
$$\lim_{n\to \infty} \frac{\int_0^1 t^{n-1} v^*_x(t) dt}{\int_0^1 t^n v_x(t)dt} = 1.$$ Since $v_1(t) > 0$
for all $t \in (0,1)$, the weights $v_x$ and $v_x^*$ satisfy \eqref{eqjm2}, so the last statement
follows from Lemma \ref{lem:weight_asymptotics} and the observation that $\lim_{t \nearrow 1} \frac{v_x(t)}{v_x^*(t)} = 1$
by elementary properties of the natural logarithm.
\end{proof}

We will now again restrict attention to absolutely continuous measures $d\mu=v(t)dt$. In this case, it makes sense to define $v_0(t)=v(t)$.
We also write
\begin{equation*}
  \widehat v(t) = v_1(t) = \int_t^1 v(x) d x.
\end{equation*}
Note that in this case $v_{x+1}(t)= \int_t^1 v_x(s)ds=\vhat_x(t) $ is valid for all $x \ge 0$, and thus the functions $v_x$ get smoother as $x$ increases. They also decay faster near 1.
The estimate in the following lemma is obvious.

\begin{lemma}\label{(1 minus t)v} If $v\in L^1[0,1]$ is positive, and $v_x$ is as above, then for all $x, \alpha >0$ we have
$v_{x+\alpha}(t)\le \frac{\Gamma(x)}{\Gamma(x+\alpha)}(1-t)^\alpha v_x(t)$  for all $t\in [0,1).$
\end{lemma}

We now investigate this procedure on the level of radial weights in the ball.
Let $\omega$ be a radial weight in $\bB_d$.
For each $x\ge 0$ we define a radial weight $\om_x$ by
\begin{equation*}
  \om_x(z) = \frac{1}{d} |z|^{2 - 2 d} \int_{|w| \ge |z|} \frac{( |w|^2-|z|^2)^{2 x - 1}}{\Gamma(2 x)} \omega(w) d V(w).
\end{equation*}
Then $\om_x$ is the radial weight that corresponds to the $L^1[0,1]$-function $v_{2x}$ that is associated with $v$ as in Lemma \ref{momentsByParts}.

\begin{theorem}\label{omxTheorem} Let $\om$ be a radial weight and let $x\ge 0$.

Then $\om_x$ is a weight,
$$\|f\|^2_{B^{-x}_\om} \approx \int_{\Bd} |f|^2 {\om_x} dV,$$ and for each $s\in \R$ we have $B^s_\om=B^{s+x}_{\om_x}$ with equivalence of norms.
\end{theorem}
\begin{proof} Since $\om$ is a radial weight,  so is $\om_x$.
  Lemma \ref{momentsByParts} implies that $n^{2x} a_n(\om_x) \approx a_n(\om)$ as $n\to \infty$. Now the Theorem follows from (\ref{Bsom-norm}).
\end{proof}
For later reference we note that  Lemma \ref{(1 minus t)v} applies  and we conclude that for all $x>0$ and $ \alpha\ge 0$

\begin{equation} \label{weightsAndGrowth}\frac{\om_{x+\alpha}(z)}{(1-|z|^2)^{2\alpha}} \le  \frac{\Gamma(2x)}{\Gamma(2x+2\alpha)} \ \om_x(z) \ \text{ for all } z\in \Bd.\end{equation}

\section{Multiplier inclusions}
\label{SectionMICradial}

\subsection{Inclusion of multiplier algebras}

Let $\omega$ be a radial weight in $\bB_d$ and let $N \in \bN$.
A crucial condition in \cite{AlHaMcCRiWeakProd} is the multiplier inclusion condition for $B^N_\om$,
which demands that
\begin{equation}
  \label{eqn:mult_inclusion}
  \Mult(B_\omega^N,B_\omega^N(\ell^2)) \subseteq
  \Mult(B_\omega^{N-1},B_\omega^{N-1}(\ell^2)) \subseteq \cdots
  \subseteq \Mult(B^0_\omega, B^0_\omega (\ell^2))
\end{equation}
with continuous inclusions.
In this Section we will show that all weighted Besov spaces defined by  radial weights satisfy this multiplier inclusion condition.
 In fact, we will prove a more general result about inclusion of the multipliers between spaces of analytic functions on the unit ball with unitarily invariant kernels.

 We first recall a few notions from the theory of operator spaces. Let $\cH$ be a Hilbert space
 and let $\cM \subseteq\cB(\cH)$ be a subspace.
 For $n \in \bN$, let $M_n(\cM)$ denote the space of all $n \times n$ matrices with entries in $\cM$.
 The natural identification of $M_n(\cB(\cH))$ with $\cB(\cH^n)$ allows us
 to endow each space $M_n(\cM)$ with a norm. Suppose now that $\cK$
 is another Hilbert space and that
 $\Phi: \cM \to \cB(\cK)$ is a linear map.
 Then for each $n \in \bN$, we obtain an induced linear map
 \begin{equation*}
   \Phi^{(n)}: M_n(\cM) \to M_n(\cB(\cK)), \quad [ m_{i j}] \mapsto [ \Phi(m_{i j})].
 \end{equation*}
 In this setting, we say that $\Phi$ is
 completely contractive if each map $\Phi^{(n)}$ is contractive.

 In Section \ref{subsection3.2} we will see that this notion has a natural analogue for operators between possibly different Hilbert spaces, and then we will mostly be interested in the case when $\cM = \Mult(\cH,\cK)$ for Hilbert function spaces $\cH$ and $\cK$.
 In this case, $M_n(\Mult(\cH,\cK))$ can be identified with $\Mult(\cH(\bC^n), \cK(\bC^n))$,
 so this approach allows us to deal with operator-valued multipliers.

 We begin with the following result, which is essentially due to Kacnelson \cite{Kacnelson72}, see also \cite[Theorem 2.1]{CP14}. For completeness, we provide a proof.
If $\cH$ is a Hilbert space with an orthogonal basis $(e_n)$,
let $\cT(\cK)$ denote the algebra of all bounded lower triangular operators
on $\cK$ with respect to $(e_n)$.

\begin{lemma}[Kacnelson]
  \label{lem:Kacnelson}
  Let $\cH$ be a Hilbert space with orthonormal basis $(e_n)$, let
  $(d_n)$ be a nonincreasing sequence of strictly positive numbers and let
  $D$ denote the diagonal operator on $\cH$ with diagonal $(d_n)$, and let $D^{-1}$ be its possibly unbounded inverse.
  Then for every $T \in \cT(\cH)$, the densely defined operator $D T D^{-1}$
  is bounded and the homomorphism
  \begin{equation*}
    \cT(\cH) \to \cT(\cH), \quad T \mapsto D T D^{-1},
  \end{equation*}
  is completely contractive.
\end{lemma}

\begin{proof}
  If $P_n$ denotes the orthogonal projection onto the linear span of $e_0,\ldots,e_n$,
  then $P_n$ commutes with every diagonal operator. Thus, a straightforward approximation
  argument shows that it suffices to prove the following assertion: For every $n \in {\mathbb N}$
  and every nonincreasing sequence of strictly positive numbers $d_0,\ldots,d_n$,
  the map
  \begin{equation*}
    \Phi: \cT_{n+1} \mapsto \cT_{n+1}, \quad T \mapsto \diag(d_0,\ldots,d_n) T \diag(d_0,\ldots,d_n)^{-1},
  \end{equation*}
  is completely contractive. Here, $\cT_{n+1}$ denotes the algebra of all lower triangular $(n+1) \times (n+1)$ matrices, and $\diag(d_0,\ldots,d_n)$ is the diagonal matrix with diagonal $d_0,\ldots,d_n$.

  To this end, let $d_0,\ldots,d_n$ be nonincreasing strictly positive numbers.
  By multiplying the sequence $d_0,\ldots,d_n$ with $d_0^{-1}$,
  we may assume that $d_0 = 1$. For $j \ge 1$, let $\alpha_j = d_j / d_{j-1}$
  and $\alpha = (\alpha_1,\ldots,\alpha_n)$.
  Then $d_j = \alpha_1 \ldots \alpha_j$ for $j \ge 1$ and $\alpha_j \in (0,1]$ by assumption.

  We will use the maximum modulus principle to show that the map $\Phi$ is completely contractive.
  For $z= (z_1,\ldots,z_n) \in (\bC \setminus \{0\})^n$, define
  \begin{equation*}
    D(z) = \diag(1,z_1,z_1 z_2,\ldots, z_1 z_2 \ldots z_n).
  \end{equation*}
  In particular, $D(\alpha) = \diag(d_0,\ldots,d_n)$.
  If $T = [t_{ij}] \in \cT_{n+1}$ and $i \ge j$, then
  the $(i,j)$-entry of $D(z) T D(z)^{-1}$ is given by
  \begin{equation*}
    z_1 z_2 \ldots z_i t_{i j} z_1^{-1} z_2^{-1} \ldots z_j^{-1}
    = t_{i j} z_{j+1} \ldots z_i.
  \end{equation*}
  Since $T$ is lower triangular, we therefore conclude that the map $z \mapsto D(z) T D(z)^{-1}$
  extends to an analytic $M_{n+1}$-valued map on $\bC^n$.

  Let $[T_{i j}] \in M_r(\cT_{n+1})$. By the maximum modulus principle,
  \begin{equation*}
    ||[\Phi(T_{i j})]|| = ||[ D(\alpha) T_{i j} D(\alpha)^{-1}]||
    \le \sup_{z \in \bT^n} || [ D(z) T_{i j} D(z)^{-1}] ||.
  \end{equation*}
  But if $z \in \bT^n$, then $D(z)$ is unitary, hence
  \begin{equation*}
    || [ D(z) T_{i j} D(z)^{-1}] || = ||(D(z) \otimes I_r) [T_{i j}] (D(z) \otimes I_r)^{-1} ||
    = ||[T_{i j}||,
  \end{equation*}
  which finishes the proof.
\end{proof}

The following corollary is merely a reformulation of Lemma \ref{lem:Kacnelson}.
\begin{corollary}
  \label{cor:kacnelson}
  Let $\cK$ be a Hilbert space with an orthonormal basis $(e_n)$. Suppose
  that $\cH$ is another Hilbert space such that $\cH \subseteq\cK$ as vector spaces,
  such that $(e_n)$ is an orthogonal basis for $\cH$ and
  such that the sequence $(||e_n||_{\cH})$ is nondecreasing.
  Then $\cT(\cH) \subseteq \cT(\cK)$, and the inclusion is a complete contraction.
\end{corollary}

\begin{proof}
  Observe that every operator in $\cT(\cH)$ is at least densely defined on $\cK$.
  Let $D$ be the diagonal operator on $\cH$ with diagonal
  $(||e_n||^{-1}_{\cH})$. Then $D$ extends to a unitary operator $\cK \to \cH$.
  Thus, if $[T_{i j}] \in M_r(\cT(\cH))$, then by Lemma \ref{lem:Kacnelson},
  \begin{equation*}
    ||[T_{i j}]||_{\cB(\cK^r)} = ||[ D T_{i j} D^{-1} ]||_{\cB(\cH^r)}
    \le ||[T_{i j}]||_{\cB(\cH^r)}.
  \end{equation*}
  This shows that $\cT(\cH) \subseteq \cT(\cK)$ completely contractively.
\end{proof}

Let $\cH$ be a reproducing kernel Hilbert space on $\bD$ with a reproducing
kernel of the form
\begin{equation*}
  k_w(z) = \sum_{n=0}^\infty a_n z \ol{w}^n,
\end{equation*}
where $a_n > 0$ for all $n \in \bN_0$. Then
\begin{equation*}
  ||z||^2_{\Mult(\cH)} = \sup_{n \in \bN_0} \frac{a_n}{a_{n+1}}.
\end{equation*}
This motivates the condition in the following result.

\begin{prop}
  \label{prop:multiplier_inclusion}
  Let $\cH$ and $\cK$ be two reproducing kernel Hilbert spaces on $\bB_d$, $d \in \bN$,
  with reproducing kernels
  $k_w(z) = \sum_{n=0}^\infty a_n \langle z, w \rangle^n$
  and
  $\ell_w(z) = \sum_{n=0}^\infty b_n \langle z, w \rangle^n$,
  respectively. Assume that $a_n,b_n > 0$ for all $n \in \bN_0$. If
  \begin{equation*}
    \frac{b_n}{b_{n+1}} \le \frac{a_n}{a_{n+1}} \quad \text{ for all } n \in \bN_0,
  \end{equation*}
  then $\Mult(\cH) \subseteq\Mult(\cK)$, and the inclusion is a complete contraction.
\end{prop}

\begin{proof}
  Observe that $\cH$ and $\cK$ each have orthonormal bases consisting of monomials.
  If we order the monomials such that their degrees are nondecreasing, then every multiplication
  operator on $\cH$ is lower triangular with respect to such an orthonormal basis.
  Moreover, if $p$ is a monomial of degree $n$ with $||p||_{\cK} = 1$, then
  \begin{equation*}
    ||p||_{\cH} = \sqrt{\frac{b_n}{a_n}}.
  \end{equation*}
  The assumption implies that the sequence $\sqrt{b_n / a_n}$ is nondecreasing.
  In particular, there exists a constant $C > 0$ such that $a_n \le C b_n$,
  so that $\cH$ is densely contained in $\cK$ and every multiplication operator
  on $\cH$ is at least densely defined on $\cK$.
  An application of Corollary \ref{cor:kacnelson}
  now shows that every multiplication operator on $\cH$ is bounded on $\cK$, and hence
  a bounded multiplication operator, and that the inclusion
  $\Mult(\cH) \subseteq\Mult(\cK)$ is a complete contraction.
\end{proof}

We obtain the following consequence for multiplier algebras of weighted Besov spaces.

\begin{corollary} \label{unrestricted} Let $\om$ be a radial weight in $\Bd$ and
  let $s, t \in \R$ with $t \le s$. Then
  \begin{equation*}
  \Mult(B^s_\om)\subseteq \Mult(B^t_\om)
  \end{equation*}
  and the inclusion is a complete contraction. In particular,
  \begin{equation*}
    \Mult(B^s_\om,B^s_\om(\ell_2))\subseteq \Mult(B^t_\om, B^t_\om(\ell_2))
  \end{equation*}
  and the inclusion is a contraction.
\end{corollary}
In particular, by taking $s=n$ and $t= n-1$ for $n=1,2.., N$ we see that any weighted Besov space $\HH=B^N_\om$ associated with a radial weight satisfies the  multiplier inclusion condition \eqref{eqn:mult_inclusion}.
\begin{proof}
  We saw in Section \ref{section_weights_basics} that $B^s_{\omega}$ and $B_{\omega}^{t}$ have
  reproducing kernels of the form
  \begin{equation*}
    k_w(z) = \sum_{n=0}^\infty a_n \langle z,w \rangle^n \quad \text{ and } \quad
    \ell_w(z) = \sum_{n=0}^\infty b_n \langle z,w \rangle^n,
  \end{equation*}
  respectively, where $a_n = ||z_1^n||^{-2}_{B^s_{\omega}}$
  and $b_n = ||z_1^n||^{-2}_{B^t_{\omega}}$. From Equation \eqref{Bsom-norm}, we deduce that for $n \ge 1$,
  \begin{equation*}
    \frac{a_n}{b_{n}} = n^{2(t - s)}
  \end{equation*}
  and $a_0 / b_0 = 1$. Since $t \le s$, the sequence $(a_n / b_n)$ is nonincreasing, so that
  the result is a special case of Proposition \ref{prop:multiplier_inclusion}.
\end{proof}

It was shown in \cite[Theorem 1.5]{AlHaMcCRiWeakProd} that the multiplier inclusion condition \eqref{eqn:mult_inclusion} for $B_\omega^N$ implies that every
bounded column multiplication operator on $B^N_\om$ is also a bounded row multiplication operator. Moreover, by Theorem \ref{omxTheorem},
each Besov space $B^s_{\om}$ can also be regarded as a space of the form $B^N_{\widetilde \omega}$ for a suitable
 radial weight $\widetilde \omega$ and $N \in \bN$.
Thus, we obtain the following consequence.

\begin{corollary}
  \label{cor:BCBR_square_integer}
  Let $\omega$ be a radial weight in $\bB_d$ and let $s \in \mathbb R$. Then
  \begin{equation*}
    M^C(B_\omega^s) \subseteq M^R(B_\omega^s)
  \end{equation*}
  and the inclusion is continuous.
\end{corollary}

We do not know if the inclusion in the preceding corollary is contractive, even in the case of the Drury--Arveson space. Even though Corollary \ref{unrestricted} shows
that the multiplier inclusion condition \eqref{eqn:mult_inclusion} holds with contractive inclusions, \cite[Theorem 1.5]{AlHaMcCRiWeakProd}
only yields boundedness of the inclusion $M^C(B_\om^N) \subseteq M^R(B_\om^N)$.

\subsection{Inclusion of multiplier spaces}\label{subsection3.2}

We also require a version of the preceding result for multipliers between different spaces.
Thus, we seek conditions that imply inclusions of the form $\Mult(\cH,\cH') \subseteq\Mult(\cK,\cK')$.
The proofs based on Kacnelson's lemma (Lemma \ref{lem:Kacnelson}) generalize to this setting. The results in
this subsection contain the results of the preceding subsection as a special case. But for the sake of readability, we chose
to treat inclusions of multiplier algebras first.

We begin with a version of Corollary \ref{cor:kacnelson} for four Hilbert spaces.
First of all, observe that if $\cH$ and $\cH'$ are Hilbert space, then $\cB(\cH,\cH')$
can be identified with a subspace of $\cB(\cH \oplus \cH')$, hence
the notion of a completely contractive map applies in this setting as well.
Equivalently, $M_r(\cB(\cH,\cH'))$ is normed by means of the identification with $\cB(\cH^r,(\cH')^r)$.
If $\cH$ and $\cH'$ are Hilbert spaces with orthogonal bases $(e_n)$ and
$(e_n')$, respectively, let $\cT(\cH,\cH') \subseteq\cB(\cH,\cH')$ denote the space of all
operators that are lower triangular with respect to $(e_n)$ and $(e_n')$. Thus,
an operator $T \in \cB(\cH,\cH')$ belongs to $\cT(\cH,\cH')$ if and only if
\begin{equation*}
  \langle T e_i, e_j' \rangle = 0 \quad \text{ whenever } j > i.
\end{equation*}

\begin{corollary}
  \label{cor:kacnelson_rectangular}
  Let $\cK$ and $\cK'$ be Hilbert spaces with orthonormal bases $(e_n)$ and $(e_n')$, respectively.
  Let $\cH$ and $\cH'$ be another pair of Hilbert spaces such that
  \begin{itemize}
    \item $\cH \subseteq\cK$ and $\cH' \subseteq\cK'$ as vector spaces,
    \item $(e_n)$ is an orthogonal basis for $\cH$ and $(e_n')$ is an orthogonal basis for $\cH'$,
    \item the sequence $( ||e_n||_{\cH})$ is nondecreasing, and
    \item $||e_n||_{\cH} \le ||e_n'||_{\cH'}$ for all $n \in \bN$.
  \end{itemize}
  Then $\cT(\cH,\cH') \subseteq\cT(\cK,\cK')$ and the inclusion is completely contractive.
\end{corollary}

\begin{proof}
  Every operator in $\cT(\cH,\cH')$ is at least a densely defined operator from $\cK$ to $\cK'$.
  Our goal is to show that these operators are bounded.

  In the proof, we will require the following  diagonal operators.
  Let $D$ be the diagonal operator on $\cH$ with diagonal $(||e_n||_{\cH}^{-1})$. Similarly,
  let $D'$ be the diagonal operator on $\cH'$ with diagonal $(||e_n'||_{\cH'}^{-1})$.
  Observe that $D$ extends to a unitary operator from $\cK$ to $\cH$ and $D'$ extends
  to a unitary operator from $\cK'$ to $\cH'$.
  Moreover, let $U \in \cB(\cH',\cH)$ be the unique unitary operator with
  \begin{equation*}
    U e_n' = \frac{||e_n'||_{\cH'}}{||e_n||_{\cH}} e_n \quad (n \in \bN).
  \end{equation*}

  Suppose now that
  $[T_{i j}] \in M_r(\cT(\cH,\cH'))$. Then by Lemma \ref{lem:Kacnelson}, we find that
  \begin{align*}
    \| [T_{i j}] \|_{\cB(\cK^r,(\cK')^r)}
    = \| [U D' T_{i j} D^{-1}] \|_{\cB(\cH^r)}
    &= \| [D D^{-1} U D' T_{i j} D^{-1}] \|_{\cB(\cH^r)} \\
    &\le \| [D^{-1} U D' T_{i j}] \|_{\cB(\cH^r)}.
  \end{align*}
  Observe that
  \begin{equation*}
    D^{-1} U D' e_n' = e_n = U S e_n' \quad (n \in \bN),
  \end{equation*}
  where $S$ is the diagonal operator on $\cH'$ with diagonal $( \frac{||e_n||_{\cH}}{||e_n'||_{\cH'}})$. By
  assumption, this operator is a contraction. From the estimate
  above and the identity $D^{-1} U D' = U S$, we infer that
  \begin{align*}
    \| [T_{i j}] \|_{\cB(\cK^r,(\cK')^r)}
    \le \| [U S T_{i j}] \|_{\cB(\cH^r)}
    &= \| [ S T_{i j}] \|_{\cB(\cH^r, (\cH')^r)} \\
    &\le || [T_{i j}]||_{\cB(\cH^r, (\cH')^r)},
  \end{align*}
  which finishes the proof.
\end{proof}

The following result is a generalization of Proposition \ref{prop:multiplier_inclusion}.

\begin{prop}
  \label{prop:rectangular_multiplier_inclusion}
  Let $d \in \bN$ and let $\cH,\cH',\cK,\cK'$ be reproducing kernel Hilbert spaces on $\bB_d$
  with respective reproducing kernels
  $k_w(z) = \sum_{n=0}^\infty a_n \langle z, w \rangle^n$,
  $k'_w(z) = \sum_{n=0}^\infty a_n' \langle z, w \rangle^n$,
  $\ell_w(z) = \sum_{n=0}^\infty b_n \langle z, w \rangle^n$ and
  $\ell'_w(z) = \sum_{n=0}^\infty b_n' \langle z, w \rangle^n$.
  Suppose that for all $n \in \bN_0$, the inequalities $a_n,a_n',b_n,b_n' > 0$ and
  \begin{equation*}
    \frac{b_n}{b_{n+1}} \le \frac{a_n}{a_{n+1}}
  \end{equation*}
  and
  \begin{equation*}
    \frac{b_n}{a_n} \le \frac{b_n'}{a_n'}
  \end{equation*}
  hold.
  Then
  \begin{equation*}
    \Mult(\cH,\cH') \subseteq\Mult(\cK,\cK'),
  \end{equation*}
  and the inclusion is completely contractive.
\end{prop}

\begin{proof}
  This follows as in the proof of Proposition \ref{prop:multiplier_inclusion} from an application of Corollary \ref{cor:kacnelson_rectangular}.
  Indeed, all four spaces have an orthogonal basis of monomials and if we order the monomials such that their degrees
  are nondecreasing, then every operator in $\Mult(\cH,\cH')$ is lower triangular. Moreover, if $p$
  is a monomial of degree $n$, then
  \begin{equation*}
    ||p||^2_{\cH} = \frac{b_n}{a_n} ||p||^2_{\cK} \quad \text{ and }
    ||p||^2_{\cH'} = \frac{b_n'}{a_n'} ||p||^2_{\cK'},
  \end{equation*}
  from which it readily follows that the last two conditions in Corollary \ref{cor:kacnelson_rectangular} hold.
  Finally, the assumptions imply that both sequences $(\frac{a_n}{b_n})$ and $(\frac{a_n'}{b_n'})$ are bounded above,
  so that $\cH \subseteq\cK$ and $\cH' \subseteq\cK'$.
\end{proof}

The last result applies in particular to the spaces $B^s_{\omega}$.

\begin{corollary}
  \label{cor:Besov_rectangular_multiplier_inclusion}
  Let $\omega$ and $\nu$ be  radial weights in $\bB_d$ and let $s,t,s',t' \in \mathbb R$ with $t \le s$ and $t'-s' \le t-s$.
Then
\begin{equation*}
  \Mult(B^s_\om, B^{s'}_{\nu}) \subseteq \Mult(B^{t}_\om, B^{t'}_{\nu})
\end{equation*}
and the inclusion is completely contractive. In particular,
\begin{equation*}
  \Mult(B^s_\om, B^{s}_{\nu}(\ell_2)) \subseteq \Mult(B^{t}_\om, B^{t}_{\nu}(\ell_2))
\end{equation*}
and the inclusion is contractive.
\end{corollary}

\begin{proof}
  We apply Proposition \ref{prop:rectangular_multiplier_inclusion} with $\cH = B^s_{\omega},
  \cH'= B^{s'}_{\nu},\cK = B^{t}_{\omega}$ and $\cK' = B^{t'}_{\nu}$.
  With notation
  as in that Proposition, the argument in the proof of Corollary \ref{unrestricted} shows that
  \begin{equation*}
    \frac{a_n}{b_n} = n^{2(t-s)}
  \end{equation*}
  for $n \ge 1$ and $a_0 / b_0 = 1$, so the sequence $(a_n/ b_n)$ is nonincreasing as $t \le s$.
  Similarly,
  \begin{equation*}
    \frac{a_n'}{b_n'} = n^{2 (t'-s')}
  \end{equation*}
  for $n \ge 1$ and $a_0' / b_0' = 1$. Since $t' - s' \le t - s$, we conclude that $a_n' / b_n' \le a_n / b_n$ for all $n \in \bN_0$,
  so the result is a special case of Proposition \ref{prop:rectangular_multiplier_inclusion}.
\end{proof}

We also obtain a multiplier space version of Corollary \ref{cor:BCBR_square_integer}.

\begin{theorem}
  \label{thm:BCBR_general}
  Let $\om$ and $\nu$ be  radial weights in $\bB_d$, and let $s, t \in \R$. Then
  \begin{equation*}
    M^C(B^s_\om, B^t_\nu) \subseteq M^R(B^s_\om,B^t_\nu)
  \end{equation*}
  and the inclusion is continuous.
\end{theorem}

\begin{proof}
  Let $\cH = B^s_{\omega}$ and $\cK = B^t_{\nu}$.
  We will use \cite[Theorem 4.2]{AlHaMcCRiWeakProd}, according to which the result follows from the
  multiplier inclusion condition for the pair $(\cH,\cK)$. To establish this property, by definition, we have to show that there are  weights
  $\widetilde \omega$ and $\widetilde \nu$ and $N \in \bN$ such that $\cH = B_{\widetilde \om}^N$, $\cK = B_{\widetilde \nu}^N$ (with equivalent norms)
  and
  \begin{equation*}
    \Mult(B_{\widetilde \omega}^N,B_{\widetilde \nu}^N(\ell^2)) \subseteq
    \Mult(B_{\widetilde \omega}^{N-1},B_{\widetilde \nu}^{N-1}(\ell^2)) \subseteq \cdots
    \subseteq \Mult(B_{\widetilde \omega}^0,B_{\widetilde \nu}^0(\ell^2))
  \end{equation*}
  with continuous inclusions.

  To this end, let $x,y \ge 0$ be real numbers such that $s+x = t + y \in \bN$ and let $N = s+x = t +y$ be this common value. Moreover,
  let $\widetilde \omega = \omega_x$ and $\widetilde \nu = \nu_y$. Then by Theorem \ref{omxTheorem}, we have $\cH = B_{\widetilde \omega}^N$
  and $\cK = B_{\widetilde \nu}^N$. The continuity of the inclusions above now follows from Corollary \ref{cor:Besov_rectangular_multiplier_inclusion},
  which concludes the proof.
\end{proof}


\section{Weakly normal weights}
\label{sec:weights_finer}

In this section, we will study several finer properties of $L^1[0,1]$ weights that will translate to Hilbert space properties
of the associated radially weighted Besov spaces.

\subsection{A doubling condition}
Recall from Section \ref{sec_v_x} that if $v \in L^1[0,1]$ is non-negative, then we defined for $x > 0$ a weight $v_x \in L^1[0,1]$ by
\begin{equation*}
  v_x(t)= \int_{[t,1]} \frac{(s-t)^{x-1}}{\Gamma(x)} v(s)ds, t \in [0,1)
\end{equation*}
and we also write $\widehat v = v_1$.
We will now discuss a class of weights on $[0,1]$, where one has an asymptotics of the type $v_{x+\alpha}(t)\approx (1-t)^\alpha v_x(t)$  for all $t\in [0,1]$ at least when $x\ge 1$.
As in \cite{PelaezRattya} we define the class $\widehat{\mathcal D}$ by saying that a non-negative integrable
 function $v$ is in  $ \widehat{\mathcal D}$ if $\vhat$ is doubling near 1, i.e. if there is a constant $c>0$ such that $\vhat(t) \le c\vhat(\frac{1+t}{2})$ for all $t\in [0,1)$. It is clear that if $v\in \widehat{\HD}$ is not identically equal to 0, then it is
a weight.   For later reference we record the following elementary lemma.

\begin{lemma}\label{doublingweights} If $v\in L^1[0,1]$ is a weight, then $v\in \widehat{\mathcal D}$ if and only if there is $M>1$ such that
$$ {\int_t^1\vhat(s) ds}\le (1-t) \vhat(t)\le  M {\int_t^1\vhat(s) ds} .$$
\end{lemma}
\begin{proof} The inequality on the left is  true for all $v\ge 0$ since $\vhat$ is nonincreasing.  First suppose that $\vhat$ is doubling. Then there is a $C>0$ such that $\vhat(t)\le C\vhat(\frac{1+t}{2})$ for all $t\in [0,1)$. Now fix $t\in [0,1)$, then
\begin{align*}
\int_t^1\vhat(s)ds&\ge \int_{t}^{(1+t)/2}\vhat(s)ds\\
&\ge  \vhat(\frac{1+t}{2}) \frac{1-t}{2}\\
&\ge  \vhat(t)\frac{1-t}{2C}.
\end{align*}

Next suppose that   there is $M>1$ such that $(1-t)\vhat(t) \le M \int_t^1\vhat(s)ds$ for all $t \in [0,1)$. Then one checks by taking a derivative that $\frac{\int_t^1\vhat(s)ds}{(1-t)^M}$ is nondecreasing and thus
\begin{align*}\frac{\vhat(t)}{(1-t)^{M-1}} &\le M\frac{\int_t^1\vhat(s)ds}{(1-t)^M}\\
&\le M\frac{\int_{\frac{1+t}{2}}^1\vhat(s)ds}{(1-\frac{1+t}{2})^M}\\
&\le M \vhat(\frac{1+t}{2})\frac{2^{M-1}}{(1-t)^{M-1}}.
\end{align*}
It follows that $\vhat$ is doubling.
\end{proof}
Notice that for $x>0$ one has $\hat{v_x}(t)=v_{x+1}(t)=\int_t^1\frac{(s-t)^{x-1}}{\Gamma(x)}\vhat(s) ds$ and with that it is easy to show that if $v\in \widehat{\HD}$, then $v_x\in \widehat{\HD}$ for each $x>0$. Then the previous  lemma along with Lemma \ref{(1 minus t)v} implies that for every  $v\in \widehat{\HD}$ we have $v_{x+1}(t) \approx (1-t)v_x(t)$ for every $x\ge 1$. Since $\vhat=v_1$ we inductively obtain $v_{n+1}(t) \approx (1-t)^n \vhat(t)$ for each $n\in \N$. But then we have for $0\le x\le n$ that
\begin{align*} 1&\ge \frac{\int_t^1\left(\frac{s-t}{1-t}\right)^x v(s)ds}{\vhat(t)}
&\ge  \frac{\int_t^1\left(\frac{s-t}{1-t}\right)^n v(s)ds}{\vhat(t)}
&=\frac{\int_t^1(s-t)^n v(s)ds}{(1-t)^n\vhat(t)}\ge C_n.
\end{align*}
Thus we have proved the following Lemma.

\begin{lemma} \label{vxdoublingweights} If $v\in \widehat{\HD}$, then for all $x\ge 0$ we have $v_{x+1}(t) \approx (1-t)^x \vhat(t)=(1-t)^x v_1(t)$.
\end{lemma}

\subsection{Weakly normal weights}

 The following definition goes back to S.N. Bernstein, \cite{Bernstein}.
\begin{definition} Let $a<b$. A function $f:[a,b)\to [0,\infty)$ is called almost decreasing if there is some $C>0$ such that $f(t) \le Cf(s)$, whenever $a \le s\le t<b$. Almost increasing is defined similarly.
\end{definition}
One reason this definition is useful for weights is the following lemma.
\begin{lemma}\label{almostDecreasing} $f:[a,b)\to [0,\infty)$ is  almost decreasing, if and only if there is a nonincreasing function $g:[a,b)\to \R$ and $c,C>0$ such that
$$cg(t)\le f(t) \le Cg(t)$$ for all $t\in [a,b)$. If $f$ is continuous, then $g$ can be chosen to be continuous as well.
\end{lemma}
\begin{proof} Suppose that $g$ is nonincreasing such that
$cg(t)\le f(t) \le Cg(t)$ for all $t\in [a,b)$. Then for  $a \le s\le t<b$ we have
$$f(t) \le Cg(t) \le Cg(s)\le \frac{C}{c} f(s)=C'f(s).$$
Conversely, suppose that $f$ is almost decreasing, then for $t\in [a,b)$ set
$$g(t)=\inf\{f(s): s\le t\}.$$
Note that if $f$ is continuous, then $g$ is continuous.
Clearly $g$ is nonincreasing and $g(t) \le f(t)$ for all $t\in [a,b)$. Furthermore, the hypothesis on $f$ implies the existence of $C>0$ such that $f(t) \le C f(s)$, whenever $a \le s\le t<b$. This implies $f(t)\le Cg(t)$ for all $t\in [a,b)$.
\end{proof}

\begin{lemma} \label{weaklyNormal} Let $v \in L^1[0,1]$ be a weight. If $t_0\in [0,1)$, $\alpha\in \R$, and $x\ge 0$ such that $\frac{(1-t)^\alpha}{v_x(t)}$ is almost decreasing in $[t_0,1)$, then so is $\frac{(1-t)^{\alpha+y}}{v_{x+y}(t)}$ for every $y\ge 0$.
\end{lemma}
\begin{proof} We consider reciprocals and thus prove a statement about almost increasing functions. One verifies that $v_{x+y}=(v_x)_y$ for all $x,y\ge 0$. Let $t, t'\in [t_0,1)$ with $t<t'$, and define $\lambda=\frac{1-t'}{1-t}$. Then $(1-\lambda)+\lambda t=t'$ and
\begin{align*}
v_{x+y}(t)
&=  \int_t^1\frac{(s-t)^{y-1}}{\Gamma(y)} \frac{v_x(s)}{(1-s)^\alpha} (1-s)^\alpha ds\\
&\le C \int_t^1\frac{(s-t)^{y-1}}{\Gamma(y)} \frac{v_x((1-\lambda)+\lambda s)}{\lambda^\alpha(1-s)^\alpha} (1-s)^\alpha ds\\
&=C \lambda^{-(\alpha+y)} \int_{t'}^1\frac{(u-t')^{y-1}}{\Gamma(y)} {v_x(u)} du\\
&= C\lambda^{-(\alpha+y)} v_{x+y}(t').
\end{align*}
The Lemma follows. \end{proof}
Now recall from  \cite{ShieldsWilliamsCrelle} that a weight function $v$ is called normal, if there are $\alpha > \beta \in \R $ such that $\frac{(1-t)^\beta}{v(t)}$ is almost increasing in $[t_0,1)$ and $\frac{(1-t)^\alpha}{v(t)}$ is almost decreasing in $[t_0,1)$ for some $0\le t_0<1$. Actually, Shields and Williams required $\beta>0$ for their results, and they wanted the limits to be  $\infty$ and 0. Furthermore, in the paper \cite{ShieldsWilliams} the ratios were assumed to be nondecreasing (resp. nonincreasing), but this definition was modified in the later paper \cite{ShieldsWilliamsCrelle}. That is a  convention that has been  used by many  authors since then.

\begin{definition} Let $\alpha\in \R$. We call a weight $v$ {\it weakly normal of order $\alpha$}, if there is $x \ge 0$ such that $\frac{(1-t)^{\alpha+x}}{v_x(t)}$ is almost decreasing in $[t_0,1)$ for some $0\le t_0<1$. The weight $v$ is called weakly normal, if it is  weakly normal of order $\alpha$ for some $\alpha \in \R$.
\end{definition}

Since $v_x$ is nonincreasing for all $x \ge 1$, we do not require an assumption corresponding to the parameter
$\beta$ above.

If a weight is weakly normal of order $\alpha$, then $\alpha>-1$. Indeed, if $v$ is weakly normal of order $\alpha$, then by Lemma  \ref{weaklyNormal} we may assume that $\frac{(1-t)^{\alpha+x}}{v_x(t)}$ is almost decreasing in $[t_0,1)$ for some $x\ge 1$ and $0\le t_0<1$. Then for $t\in [t_0,1)$ we have
$$(1-t)^{\alpha+x} \le C v_x(t) \le \frac{C}{\Gamma(x)} (1-t)^{x-1}\vhat(t).$$ But we have $\vhat(t)\to 0$ as $t\to 1$. We see that this is only possible if $\alpha>-1$.

Obviously  $v(t)= (1-t)^\alpha$ is weakly normal of order $\alpha$, whenever $\alpha >-1$.
It is also clear from the identity $v_{x+y}=(v_x)_y$ and Lemma \ref{weaklyNormal} that $v$ is weakly normal, if and only if $v_x$ is weakly normal for each $x\ge 0$, and this happens if and only if $v_x$ is weakly normal for some $x\ge 0$. In the following Lemma we have summarized the relationship of the weakly normal weights with the class $\widehat{\HD}$ and with another class  of weights that has been considered in the literature. For $\eta >-1$ the Bekoll\'e-Bonami class $B_2(\eta)$ is defined by
$$ \frac{v(t)}{(1-t)^\eta}\in B_2(\eta) \ \Longleftrightarrow \ \int_t^1v(s)ds \ \int_t^1 \frac{(1-s)^{2\eta}}{v(s)}ds \approx (1-t)^{2\eta+2}.$$
This is the radial weight version of a more general definition that characterizes the weights $\om$ on $\Bd$ such that a corresponding Bergman projection is bounded on $L^2(\om)$, see e.g.  \cite{Bekolle}.

 \begin{lemma}\label{BekolleAndDoubling}
   Let $v \in L^1[0,1]$ be a weight.

   (a)  If $\eta>-1$ and $\frac{v(t)}{(1-t)^\eta}\in B_2(\eta)$, then  $v$ is weakly normal of order $2\eta+1$.

 (b)  Let $v\in L^1[0,1]$ be non-negative. Then the following are equivalent: \begin{enumerate}
 \item $v$ is weakly normal,
\item there are $x\ge 0$ and $\eta>-1$ such that $\frac{v_x(t)}{(1-t)^\eta}\in B_2(\eta)$,
\item there is  $x\ge 0$  such that $v_x\in \widehat{\HD}$.
\end{enumerate}
 \end{lemma}

 \begin{proof} (a) Let $\eta>-1$ and suppose $\frac{v(t)}{(1-t)^\eta}\in B_2(\eta)$, then $g(t)= \int_t^1\frac{(1-s)^{2\eta}}{v(s)} ds$ is nonincreasing and the hypothesis implies that $$g(t) \approx \frac{(1-t)^{2\eta+2}}{v_1(t)}.$$ Thus Lemma \ref{almostDecreasing} implies that $\frac{(1-t)^{2\eta+2}}{v_1(t)}$ is almost decreasing, i.e. $v$ satisfies the definition of weakly normal of order $2\eta+1$ with $x=1$.

(b) $(ii)\Rightarrow (i)$ follows from (a) and the earlier observation that $v$ is weakly normal if and only if $v_x$ is weakly normal for some $x \ge 0$.

$(iii)\Rightarrow (ii)$ By Lemma \ref{doublingweights} $v_x\in \widehat{\HD}$ if and only if there is a $C> 1$ such that $(1-t)v_{x+1}(t) \le Cv_{x+2}(t)$. By use of a first derivative one sees that this is equivalent to $\frac{(1-t)^C}{v_{x+2}(t)}$ being nonincreasing. Hence
$$\int_t^1v_{x+2}(s)ds \int_t^1 \frac{(1-s)^C}{v_{x+2}(s)} ds \le v_{x+2}(t)(1-t)  \frac{(1-t)^C}{v_{x+2}(t)} (1-t)= (1-t)^{C+2}.$$ Thus $\frac{v_{x+2}(t)}{(1-t)^{C/2}} \in B_2(C/2)$.

$(i)\Rightarrow (iii)$ If $v$ is weakly normal, then there are $x \ge 1$ and $\alpha >-1$ such that $\frac{(1-t)^{\alpha+x}}{v_x(t)}$ is almost decreasing. Then one easily checks directly that $v_x \in \widehat{\HD}$.
 \end{proof}

Weights of the type $(1-t)^\alpha \left(\frac{1}{t}\log\frac{1}{1-t}\right)^\beta$ for $\alpha >-1$, $\beta\ge 0$ are weakly normal of order $\alpha$. If $\beta<0$, then such a weight would be weakly normal of order $\gamma$ for each $\gamma >\alpha$. This also holds when $\alpha=-1$, although for $\beta<-1$ the weight $\frac{\left(\frac{1}{t}\log\frac{1}{1-t}\right)^\beta}{1-t}$ is not a Bekoll\'e weight.

Part (b) of the previous lemma could be paraphrased by saying that the weakly normal weights could also have been called "weakly doubling" or "weak Bekoll\'{e} weights". For us the viewpoint of weakly normal is important, because the order of a weakly normal weight determines the  cut-off for a weighted Besov space to have the Pick property, see Theorem \ref{PickandBekolle}. The following Theorem is instrumental for the proof.

\begin{theorem} \label{existenceOfMu} Let $v\in L^1[0,1]$ be a weight.  If  $v$ is weakly normal of order $\alpha>-1$,
then there is a positive Borel measure $\mu$ on $[0,1]$ such that
$$ \int_0^1 t^n v(t)dt \int_{[0,1]} t^n d\mu(t) \approx n^{-\alpha-1} \ \text{ as }n \to \infty.$$
\end{theorem}

\begin{proof} By Lemma \ref{momentsByParts} it will suffice to show that for some $x\ge 0$ there is a measure $\mu$ with
$$ \int_0^1 t^n v_x(t)dt \int_{[0,1]} t^n d\mu(t) \approx n^{-\alpha-x-1} \ \text{ as }n \to \infty.$$

Note that $v_x$ is continuous for all $x\ge 1$. Thus, by the hypothesis and Lemmas \ref{weaklyNormal} and \ref{almostDecreasing}  there is $x\ge 1$ and a nonincreasing continuous function $g$ on $[0,1)$ such that $g(t) \approx \frac{(1-t)^{\alpha+x}}{v_x(t)}$ for $t\in [t_0,1)$.

  By Lemmas \ref{weaklyNormal} and \ref{BekolleAndDoubling} we may assume that $x\ge 1$,  $v_{x-1}\in \widehat{\HD}$,  and hence that $v_x$ is nonincreasing.

 We set $g(1)=\lim_{t\to 1}g(t)$. Then there is a Borel measure $\mu$ on $[0,1]$ such that $g(t)= \mu([t,1])$. Note that $g\in L^1[0,1]$,
 $$\hat{g}(t) =\int_t^1 g(s) ds \approx \int_t^1 \frac{(1-s)^{\alpha+x}}{v_x(s)} ds\le \int_t^1 \frac{(1-s)^{\alpha+x}}{v_x(\frac{1+s}{2})} ds , \ t\in [t_0,1).$$
 Hence $$\hat{g}(t) \lesssim 2^{\alpha+x+1} \int_{\frac{1+t}{2}}^1 \frac{(1-u)^{\alpha+x}}{v_x(u)} du \approx \hat{g}(\frac{1+t}{2}).$$
 This implies that $g\in \widehat{\HD}$.

 Since $v_{x-1}\in \widehat{\HD}$ we also have $v_{x}\in \widehat{\HD}$.  Thus Lemma A of \cite{PelaezRattya} with $v_x, g\in \widehat{\HD}$ implies that $\int_0^1 t^n g(t)dt \approx \int_{1-\frac{1}{n}}^1 g(t) dt$ and $\int_0^1 t^n v_x(t)dt \approx \int_{1-\frac{1}{n}}^1 v_{x}(t)dt$. Furthermore, by Lemma \ref{doublingweights} applied with $v_{x-1}$ we have $\int_{t}^1 v_{x}(s)ds\approx (1-t)v_x(t)$. Also noting that Lemma \ref{momentsByParts} implies that $\int_0^1t^n d\mu \approx n \int_0^1 t^ng(t)dt$ we obtain
 \begin{align}\label{momentsOfvx}
 \int_0^1 t^n v_x(t)dt \int_{[0,1]} t^n d\mu(t) &\approx v_x(1-\frac{1}{n}) \int_{1-\frac{1}{n}}^1\frac{(1-s)^{\alpha+x}}{v_x(s)}ds.
 \end{align}
 Since $v_x$ is nonincreasing we immediately obtain
 \[
 \int_0^1 t^n v_x(t)dt \int_{[0,1]} t^n d\mu(t) \ge c \int_{1-\frac{1}{n}}^1(1-s)^{\alpha+x}ds \approx n^{-\alpha-x-1}.
 \]

 By the hypothesis the ratio $\frac{v_x(t)}{(1-t)^{\alpha+x}}$ is almost increasing, hence there is  $C>0$ such that for $s\ge  1-\frac{1}{n}$
 $$v_x(1-\frac{1}{n}) \le C n^{-\alpha-x} \frac{v_x(s)}{(1-s)^{\alpha+x}}. $$ Thus we may substitute this inequality into (\ref{momentsOfvx}), and this concludes the proof of the Theorem.
\end{proof}

We will say that a radial weight $\om$ on $\Bd$ is weakly normal (of order $\alpha>-1$), if the associated $L^1[0,1]$-function $v$ (see Section \ref{section_weights_basics}) is weakly normal (of order $\alpha>-1$). For weakly normal radial weights Lemmas \ref{vxdoublingweights} and \ref{BekolleAndDoubling} imply that there is $x_0 \ge 0$ such that $\om_{x+x_0}\approx (1-|z|^2)^{2x}\om_{x_0}$ for all $x\ge 0$.

 \begin{example}\label{ExpExample} Examples of weights that are not weakly normal are $\om(z)=(1-|z|^2)^\beta e^{\frac{-1}{1-|z|^2}}$,  $\beta \in \R$. One checks with Lemma \ref{almostDecreasing} that such a weight would be weakly normal, if and only if $v(t)= (1-t)^\beta e^{\frac{-1}{1-t}}$ is weakly normal. We calculate
\begin{align*} (1-t)^2v(t)&= -\int_t^1 \frac{d}{ds} (1-s)^{\beta+2} e^{\frac{-1}{1-s}}ds\\
&=\int_t^1 (1+(\beta+2)(1-s)) v(s)ds\\
&\approx \int_t^1 v(s)ds=\vhat(t)
\end{align*}
Iteration of this shows that for each positive integer $N$ we have $(1-t)^{2N}v(t) \approx v_N(t)$ and hence $\om_N \approx (1-|z|^2)^{4N}\om$.   For more on such weights, see \cite{PavlovicPelaez}.
\end{example}

\section{Radial weights and complete Pick spaces}
\label{SectionRadialPick}
Our result on radially weighted Besov spaces that are complete Pick spaces is based on the following Lemma.
\begin{lemma}\label{Pick kernel} Let $\mu$ be a probability measure on $[0,1]$. Then there are $c_n\ge 0$ such that
$$\int_{[0,1]} \frac{1}{1-tz} d\mu(t)= \frac{1}{1-\sum_{n=1}^\infty c_n z^n}$$ for all $|z|<1$.
\end{lemma}
It follows that $k_w(z)=\int_{[0,1]} \frac{1}{1-t\la z, w\ra} d\mu(t)$ defines a normalized Pick kernel in $\Bd$.
\begin{proof} Let $F(s)=\int_{[0,1]} t^s d\mu(t)$ be the moment generating function for this set-up. It is well-known that $\log F(s)$ defines a convex function on $[0,\infty)$. In fact, it easily follows from H\"{o}lder's inequality that $F(\lambda s_1 +(1-\lambda)s_2)\le F(s_1)^\lambda F(s_2)^{1-\lambda}$ for all $s_1,s_2\in [0,\infty)$ and $0<\lambda<1$. The logarithmic convexity of $F$ follows from this. Thus for each $n\ge 0$ we have $$\log F(n+1)-\log F(n) \le \log F(n+2)-\log F(n+1),$$ which is equivalent to $\frac{F(n+1)}{F(n)}$ being nondecreasing in $n$. Now the conclusion of the lemma follows from Kaluza's lemma (see e.g. \cite{AgMcC}, Lemma 7.38) since
$\int_{[0,1]} \frac{1}{1-tz} d\mu(t)=\sum_{n=0}^\infty F(n) z^n$.
\end{proof}

\begin{theorem} \label{PickandBekolle} If $\om$ is a weakly normal radial weight  of order $\alpha >-1$, then $B^s_\om$ is a complete Pick space for all $s \ge \frac{\alpha+d}{2}$.
\end{theorem}

\begin{proof} Let $v$ be the $L^1[0,1]$-function associated with $\om$ and set $\alpha'=2s-d\ge \alpha$. Then $v$ is a weakly normal weight of order $\alpha'$.

For $n\in \N_0$ let $a_n=a_n(\om)= \int_0^1 t^n v(t)dt$, and choose a probability measure $\mu$ such that $b_n=\int_{[0,1]}t^n d\mu\approx \left(n^{\alpha'+1}{a_n}\right)^{-1}= \left(n^{2s-d+1}{a_n}\right)^{-1}.$ This can be done by Theorem \ref{existenceOfMu}. Define
$$k_w(z) = \int_{[0,1]} \frac{1}{1-t\la z, w\ra} d\mu(t).$$
Then $k_w(z)=\sum_{n=0}^\infty b_n \la z,w\ra^n$ is a normalized complete Pick kernel by Lemma \ref{Pick kernel}. Let $\HH$ be the reproducing kernel Hilbert space with kernel $k$, and let $\|f\|_{H^2_d}$ denote the Drury-Arveson norm of a function $f=\sum_n f_n$. It is easy to check and well-known that
$\|f\|^2_{\HH}= \sum_{n=0}^\infty \frac{1}{b_n}\|f_n\|^2_{H^2_d}$.
Recall that for  homogeneous polynomials $f_n$ of degree $n$ we have $$\|f_n\|^2_{H^2_d} = c_n \|f_n\|^2_{H^2(\dB)}, \text{ where }c_n \approx (n+1)^{d-1},$$
see for example formula (2.2) of \cite{RiSunkes}.
We now apply the above and the definition of the $B^s_\om$-norm to obtain
\begin{align*} \|f\|^2_{B^s_\omega}&=|f(0)|^2 + \sum_{n=1}^\infty n^{2s}a_n\|f_n\|^2_{H^2(\dB)}\\
&\approx |f(0)|^2+\sum_{n=1}^\infty n^{2s-d+1}a_n\|f_n\|^2_{H^2_d}\\
&\approx |f(0)|^2+ \sum_{n=1}^\infty \frac{1}{b_n}\|f_n\|^2_{H^2_d}\\
&= \|f\|^2_{\HH}.
\end{align*}
\end{proof}

\begin{corollary}   If $\om$ is a weakly normal radial weight  of order $\alpha >-1$, then for every $s_0\ge (\alpha+d)/2$, there is a positive nonincreasing continuous function $g\in \widehat{\HD}$ such that for every $x,y \ge 0$
$$k_w(z) =\int_0^1 \frac{(1-t)^x}{(1-t\la z, w\ra)^{x+3+2y}} \hat{g}(t) dt$$ is a reproducing kernel for $B^{s_0-y}_\om$.
\end{corollary} By this we mean that there is an alternate norm on $B^{s_0-y}_\om$ which is equivalent to the natural norm and such that $k_w(z)$ is the reproducing kernel for the space under the alternate norm.
\begin{proof} Since $s_0\ge (\alpha+d)/2$ Theorem \ref{PickandBekolle} implies that the space $B^{s_0}_\om$ is a complete Pick space. Furthermore, the proof of Theorem \ref{PickandBekolle} shows that $$k^{s_0}_w(z) = \int_0^1\frac{1}{1-t \la z, w\ra}d\mu(t)=\sum_{n=0} \la z, w\ra^n \int_{[0,1]}t^n d\mu(t)$$ is a reproducing kernel for $B^{s_0}_\om$. The existence of the measure $\mu$ was established by means of Theorem \ref{existenceOfMu}, whose proof  shows that $\mu$ can be chosen so that $g(t)=\mu([t,1])$ is continuous and satisfies $g\in \widehat{\HD}$. For $x\ge 0$ let $w_x$ be the $L^1[0,1]$-function associated with $\mu$ as in Lemma \ref{momentsByParts}, then $w_1=g,$ $w_2=\hat{g}$, and Lemma \ref{vxdoublingweights} implies that $w_{x+2}(t) \approx (1-t)^x \hat{g}(t)$. Now consider the power series
$$k_w(z) =\int_0^1 \frac{(1-t)^x}{(1-t\la z, w\ra)^{x+3+2y}} \hat{g}(t) dt= \sum_{n=0}^\infty a_n \la z, w\ra ^n,$$
where \begin{align*} a_n &\approx (n+1)^{x+2+2y} \int_0^1t^n(1-t)^x \hat{g}(t) dt\\
&\approx (n+1)^{x+2+2y} \int_0^1t^n w_{x+2}(t)dt\\
& \approx (n+1)^{2y} \int_{[0,1]}t^n d\mu(t) \end{align*}
by Lemma \ref{momentsByParts}.
It is easy to see that if $k^{s_0}_w(z)$ is a reproducing kernel for $B^{s_0}_\om$, then $k_w(z)$ is a reproducing kernel for $B^{s_0-y}_\om$.
\end{proof}
\begin{corollary}\label{kernelgrowth} Let $\om$ be a weakly normal radial weight on $\Bd$. For $s\in \R$ let $k^s_w(z)$ be the reproducing kernel for $B^s_\om$.

Then for each $s\le t$ there is $c>0$ such that $k^s_z(z) \le c\frac{k^t_z(z)}{(1-|z|^2)^{2(t-s)}}$ for all $z\in \Bd$.
\end{corollary} \begin{proof}  If $v$ is weakly normal normal of order $\alpha>-1$, then choose $s_0 \ge \max(t, (\alpha+d)/2)$. Then by the previous corollary with $x=0$ we have
$$k_z^s(z) \approx \int_0^1 \frac{\hat{g}(u)}{(1-u|z|^2)^{3+2(s_0-s)}}du$$ and  $$k_z^t(z) \approx \int_0^1 \frac{\hat{g}(u)}{(1-u|z|^2)^{3+2(s_0-t)}}du.$$ The Corollary follows from this.
\end{proof}



\section{Further results about multipliers of $B^s_\om$}
\label{SectionFurther}
Some of the main results from the previous sections are about bounded column operators on weighted Besov spaces with  radial weights. In this section we collect more facts about such operators.

Let $\alpha \ge 0$ be a real parameter. We will need to use the growth space $A^{-\alpha}(\ell_2)$ defined by
$$A^{-\alpha}(\ell_2) = \{\Phi=(\varphi_1, \varphi_2,...), \varphi_i\in \Hol(\Bd), \|\Phi\|_{A^{-\alpha}(\ell_2)}<\infty \},$$ where
$$\|\Phi\|^2_{A^{-\alpha}(\ell_2)}=
\sup_{z\in \Bd} (1-|z|^2)^{2\alpha} \sum_{i=1}^\infty |\varphi_i(z)|^2.$$

If $\alpha=0$, then we just obtain the bounded analytic functions and we observe $H^\infty(\C,\ell_2)=A^{0}(\ell_2)$ and $\|\Phi\|_\infty=\|\Phi\|_{A^{0}(\ell_2)}$.

The following lemma is well-known.
\begin{lemma}\label{GrowthAndDerivative} Let $\gamma>0$, $n\in \N$. Then there is a $c>0$ such that for all sequences of analytic functions $\Phi=(\varphi_1,\varphi_2,...)$ on $\Bd$ we have
$$\frac{1}{c}\|\Phi\|_{A^{-\gamma}(\ell_2)} \le \|\Phi(0)\|_{\ell_2}+ \|R^n \Phi\|_{A^{-\gamma-n}(\ell_2)} \le c \|\Phi\|_{A^{-\gamma}(\ell_2)}.$$
and hence $\Phi \in A^{-\gamma}(\ell_2)$ if and only if $R^n \Phi \in A^{-\gamma-n}(\ell_2)$
and $\Phi(0) \in \ell_2$.

Furthermore, if $\Phi \in H^\infty(\C,\ell_2)$, then $R^n\Phi \in A^{-n}(\ell_2)$ and $$\|R^n \Phi\|_{A^{-n}(\ell_2)} \le c \|\Phi\|_{H^\infty}. $$ \end{lemma}

\begin{proof} By induction it follows that it suffices to prove the case where $n=1$. Furthermore, that case follows easily from the formulas
$\varphi(z)= \varphi(0) + \int_0^1 R\varphi(tz) \frac{dt}{t}$ and $R\varphi(z) =\frac{1}{2\pi i} \int_{|\lambda -1|=r}\frac{\varphi(\lambda z)}{(\lambda-1)^2} d \lambda$, $r=(1-|z|)/2$.
\end{proof}

\begin{theorem} \label{Multiplier-growth-derivatives} Let $\om$ be a radial weight, let    $s, t \in \R$ with $t\le s$, and let $\Phi \in A^{-(s-t)}(\ell_2)$.

Then the following are equivalent:

(a) $\Phi \in \Mult(B^s_\om,B^t_\om({\ell_2}))$,

(b) there exists $n\in \N_0$ such that $R^n\Phi \in \Mult(B^s_\om,B^t_{\om_n}({\ell_2}))$,

(c) for all $n\in \N_0$ we have $R^n\Phi \in \Mult(B^s_\om,B^t_{\om_n}({\ell_2}))$.

In fact, for each $n\in \N$ we have $$\|\Phi\|_{A^{-(s-t)}(\ell_2)}+ \|\Phi\|_{\Mult(B^s_\om,B^t_\om({\ell_2}))} \approx \|\Phi\|_{A^{-(s-t)}(\ell_2)}+ \|R^n\Phi\|_{\Mult(B^s_\om,B^t_{\om_n}({\ell_2}))}.$$
\end{theorem}

\begin{proof} Let $n\in \N_0$.  The equivalence of the three conditions and the equivalence of norms will follow from an obvious inductive argument once we show the two inequalities \begin{align}\label{R^n+1}\|R^{n+1}\Phi\|_{\Mult(B^s_\om,B^t_{\om_{n+1}}({\ell_2}))} \lesssim  \|R^{n}\Phi\|_{\Mult(B^s_\om,B^t_{\om_{n}}({\ell_2}))} \\
\label{R^n}\|R^n\Phi\|_{\Mult(B^s_\om,B^t_{\om_n}({\ell_2}))} \lesssim \|\Phi\|_{A^{-(s-t)}(\ell_2)}+\|R^{n+1}\Phi\|_{\Mult(B^s_\om,B^t_{\om_{n+1}}({\ell_2}))} .\end{align}

Let $R^n\Phi \in \Mult(B^s_\om,B^t_{\om_n}({\ell_2}))$. It follows from Corollary \ref{cor:Besov_rectangular_multiplier_inclusion}
that $$\|R^n\Phi\|_{\Mult(B^{s-1}_{\om},B^{t-1}_{\om_{n}}({\ell_2}))}\le \|R^{n}\Phi\|_{\Mult(B^s_\om,B^t_{\om_n}({\ell_2}))}.$$

Since $B^{s-1}_{\om}= B^{s}_{\om_1}$ and $B^{t-1}_{\om_{n}}=B^{t}_{\om_{n+1}}$ with equivalence of norms by Theorem \ref{omxTheorem}, we conclude that  for $h\in B^s_\om$
\begin{align*} &\|(R^{n+1}\Phi)h\|_{B^t_{\om_{n+1}}({\ell_2})} \\
  \le &\|R((R^{n}\Phi)h)\|_{B^t_{\om_{n+1}}({\ell_2})}+ \|(R^{n}\Phi)Rh\|_{B^t_{\om_{n+1}}({\ell_2})}\\
\lesssim &\|(R^{n}\Phi)h\|_{B^t_{\om_{n}}({\ell_2})}+ \|R^{n}\Phi\|_{\Mult(B^s_{\om_1},B^t_{\om_{n+1}}({\ell_2}))} \|Rh\|_{B^s_{\om_{1}}}\\
\lesssim  &2 \|R^{n}\Phi\|_{\Mult(B^s_\om,B^t_{\om_n}({\ell_2}))} \|h\|_{B^s_{\om}}.\end{align*}
Thus (\ref{R^n+1}) holds and $R^{n+1}\Phi \in \Mult(B^s_\om,B^t_{\om_{n+1}}({\ell_2}))$.

Next we assume that $R^{n+1}\Phi \in \Mult(B^s_\om,B^t_{\om_{n+1}}({\ell_2}))$, we write
\begin{equation*}
  M_{n+1}(\Phi) =
  \|R^{n+1}\Phi\|_{\Mult(B^s_\om,B^t_{\om_{n+1}}({\ell_2}))} + \|(R^n \Phi)(0)\|_{\ell_2}
\end{equation*}
and we choose an integer $N\ge s$.
Let $k$ be an integer with $0\le k\le N$. Since $B^s_{\om_k} = B^{s-k}_\om$ and $B^t_{\om_{n+1+k}} = B^{t-k}_{\om_{n+1}}$ with equivalence of norms, Corollary \ref{cor:Besov_rectangular_multiplier_inclusion} applied to the function $R^{n+1}\Phi$
implies that $$\|R^{n+1}\Phi\|_{\Mult(B^s_{\om_k},B^t_{\om_{n+1+k}}({\ell_2}))}\le M_{n+1}(\Phi).$$

Then for all $h\in B^s_\om$ we have
\begin{align*}
  &\|(R^{n}\Phi)h\|_{B^t_{\om_{n}}({\ell_2})} \\ \lesssim &\|(R^n\Phi)(0)h(0)\|_{{\ell_2}}+\|R((R^{n}\Phi)h)\|_{B^t_{\om_{n+1}}({\ell_2})}\\
  \lesssim &M_{n+1}(\Phi) \|h\|_{B^s_\om} +  \|(R^{n+1}\Phi)h\|_{B^t_{\om_{n+1}}({\ell_2})}+ \|(R^{n}\Phi) Rh\|_{B^t_{\om_{n+1}}({\ell_2})}\\
 \lesssim &2M_{n+1}(\Phi)\|h\|_{B^s_\om}+  \|R((R^{n}\Phi) Rh)\|_{B^t_{\om_{n+2}}({\ell_2})}\\
 \lesssim &2M_{n+1}(\Phi)\|h\|_{B^s_\om}+  \|(R^{n+1}\Phi) Rh\|_{B^t_{\om_{n+2}}({\ell_2})} +\|(R^{n}\Phi) R^2h\|_{B^t_{\om_{n+2}}({\ell_2})}\\
 \lesssim &2M_{n+1}(\Phi)\|h\|_{B^s_\om}+  M_{n+1}(\Phi)\| Rh\|_{B^s_{\om_{1}}} +\|(R^{n}\Phi) R^2h\|_{B^t_{\om_{n+2}}({\ell_2})}\\
 \lesssim &3M_{n+1}(\Phi)\|h\|_{B^s_\om}+   \|(R^{n}\Phi) R^2h\|_{B^t_{\om_{n+2}}({\ell_2})}.
 \end{align*}
Thus iteration of this argument shows that
\begin{align*}\|(R^{n}\Phi)h\|_{B^t_{\om_{n}}({\ell_2})} &\lesssim (N+1)M_{n+1}(\Phi)\|h\|_{B^s_\om}+ \|(R^{n}\Phi) R^{N}h\|_{B^t_{\om_{n+N}}({\ell_2})}.\end{align*}
Since $M_{n+1}(\Phi)$ is dominated by the right-hand side of \eqref{R^n}, it remains to estimate the second summand.
Note that as $n+N\ge  t$ we have $B^t_{\om_{n+N}}({\ell_2})=L^2_a(\om_{n+N-t},{\ell_2})$ with equivalence of norms.
The growth hypothesis on $\Phi$ and Lemma \ref{GrowthAndDerivative} imply that
$R^n\Phi \in  A^{-(s-t+n)}({\ell_2})$ with $\|R^n \Phi\|_{A^{-(s-t+n)}} \lesssim \|\Phi\|_{A^{-(s-t)}}$, so using \eqref{weightsAndGrowth}, we see that
\begin{align*}\|(R^{n}\Phi)R^{N}h\|^2_{B^t_{\om_{n+N}}({\ell_2})}&\approx \int_{\Bd} \|(R^{n}\Phi)(z) R^{N}h(z)\|^2_{{\ell_2}}\ \om_{n+N-t}dV \\
  &\lesssim \|R^n \Phi\|^2_{A^{-(s-t+n)}} \int_{\bB_d} | R^N h(z)|^2 \frac{\omega_{n+N-t}}{(1 - |z|^2)^{2 (s - t + n)}} d V \\
  &\lesssim \|R^n \Phi\|^2_{A^{-(s-t+n)}} \int_{\bB_d} | R^N h(z)|^2 \om_{N-s} d V \\
&\lesssim \|\Phi\|^2_{A^{-(s-t)}(\ell_2)} \|R^N h\|^2_{L^2_a(\om_{N-s})}\\
&\lesssim \|\Phi\|^2_{A^{-(s-t)}(\ell_2)} \|h\|^2_{B^s_\om}.\end{align*}
Thus (\ref{R^n}) holds and this concludes the proof.\end{proof}

Since the multipliers of a space into itself are always bounded we obtain an immediate consequence:
\begin{theorem}  \label{MultiplierAndCarleson}   Let $\om$ be a radial weight in $\Bd$, and let $s\in \R$, $N\in \N_0$. Then
$$\Mult(B^s_\om, B^s_\om(\ell_2))=\{\Phi\in  H^\infty(\C,\ell_2): R^N\Phi \in \Mult(B^s_\om, B^{s-N}_{\om}(\ell_2))$$ and
$\|\Phi\|_{\Mult(B^s_\om, B^s_\om(\ell_2))} \approx \|R^N\Phi\|_{\Mult(B^s_\om,B^{s-N}_\om(\ell_2))} + \|\Phi\|_\infty.$
\end{theorem}
Note that if $N\ge s$, then $B^{s-N}_\om=L^2_a(\om_{N-s})$ is a weighted Bergman space and the condition in the Corollary says that the higher order derivatives of multipliers satisfy a Carleson measure condition that is appropriate for the space $B^s_\om$ (see \cite{CasFabOrt}):
There exists a $c>0$ such that $$\int_{\Bd} |f(z)|^2 \|R^N \Phi(z)\|_{\ell_2}^2 \om_{N-s}(z) dV \le c \|f\|^2_{B^s_\om}$$ for all $f \in B^s_\om$.

For the standard weights $\om(z) =(1-|z|^2)^\eta$, $\eta >-1$ the scalar case of this theorem is due to Fabrega and Ortega and \cite{OrtFab}, also see \cite{CasFabOrt}. For general Bekoll\'e-Bonami weights (not necessarily radial) it is in \cite{CasFab_Bekolleweights}. We note that in those contexts the $L^p$-case was treated as well.

We don't know whether  the equivalence of (b) and (c) of Theorem \ref{Multiplier-growth-derivatives} for $n\ge 1$ remains true without the hypothesis that  $\Phi \in A^{-(s-t)}(\ell_2)$. For general Bekoll\'e-Bonami weights that is the case, see \cite{CasFab_Bekolleweights}. We will now see that it also holds for all weakly normal radial weights.

\begin{lemma} \label{weaklyNormalMultiplierGrowth} Let $\om$ be a weakly normal radial weight, and let $s, \alpha \in \R$ with $\alpha \ge 0$. Then
$$\Mult(B^s_\om, B^{s-\alpha}_\om(\ell^2))\subseteq A^{-\alpha}(\ell_2).$$
\end{lemma}
\begin{proof} Let $\Phi=\{\varphi_1,...\}\in \Mult(B^s_\om, B^{s-\alpha}_\om(\ell^2))$. Then for each $z\in \Bd$ we have
\begin{align*} \sum_{n\ge 1} |\varphi_n(z)|^2 &=\sum_{n\ge 1} \frac{|\la \varphi_nk^s_z, k_z^{s-\alpha}\ra|^2}{|k^s_z(z)|^2}\\
&\le \frac{\| \Phi\|^2_{\Mult(B^s_\om, B^{s-\alpha}_\om(\ell^2))}  \|k_z^s\|^2 \| k_z^{s-\alpha}\|^2}{|k^s_z(z)|^2}\\
&= \| \Phi\|^2_{\Mult(B^s_\om, B^{s-\alpha}_\om(\ell^2))} \frac{\| k_z^{s-\alpha}\|^2}{\|k^s_z\|^2}\\
&\le c\| \Phi\|^2_{\Mult(B^s_\om, B^{s-\alpha}_\om(\ell^2))} (1-|z|^2)^{-2\alpha} \ \text{ by Corollary \ref{kernelgrowth}.}
\end{align*}
\end{proof}
\begin{corollary} Let $\om$ be a weakly normal radial weight, and let $s, t \in \R$ with $t<s$.
Then the following are equivalent:

(a) $\Phi \in \Mult(B^s_\om,B^t_\om({\ell_2}))$,

(b) there exists $n\in \N_0$ such that $R^n\Phi \in \Mult(B^s_\om,B^t_{\om_n}({\ell_2}))$,

(c) for all $n\in \N_0$ we have $R^n\Phi \in \Mult(B^s_\om,B^t_{\om_n}({\ell_2}))$.
\end{corollary}
\begin{proof} This follows from Theorem \ref{Multiplier-growth-derivatives}, because by Lemmas \ref{weaklyNormalMultiplierGrowth} and \ref{GrowthAndDerivative} each of the cases (a), (b), or (c) automatically implies the  required growth hypothesis. \end{proof}

In particular, applying this to $R\Phi$ and $t=s-1$, we conclude that for weakly normal radial weights we have $R\Phi\in \Mult(B^s_\om, B^{s-1}_\om(\ell_2))$ if and only if there is $n\in \N$ such that $R^n\Phi\in \Mult(B^s_\om, B^{s-n}_\om(\ell_2))$.

\bibliography{Biblio_WeakProduct_Quotients}

 \end{document}